\documentclass[12pt]{amsart}
\usepackage{color}
\usepackage{amsmath,amsfonts,amsthm,amssymb} %%amssymb not needed here
\usepackage{graphicx}
\usepackage{comment}
\usepackage{hyperref}
\usepackage{cite}
\usepackage{enumerate}
\newtheorem{theorem}{Theorem}
\newtheorem{lemma}[theorem]{Lemma}
\newtheorem{corollary}[theorem]{Corollary}
\newtheorem{proposition}[theorem]{Proposition}
\theoremstyle{definition}
\newtheorem{remark}[theorem]{Remark}

\newtheorem{definition}[theorem]{Definition}

\usepackage[us,12hr]{datetime}
\numberwithin{equation}{section} \numberwithin{theorem}{section}
\newcounter{stepctr}
{\end{list}}

\def\XXint#1#2#3{{\setbox0=\hbox{$#1{#2#3}{\int}$}
     \vcenter{\hbox{$#2#3$}}\kern-.5\wd0}}

\newcommand{\mc}[1]{\mathcal{#1}}
\newcommand{\mbb}[1]{\mathbb{#1}}

\DeclareMathOperator{\tr}{tr} 

\DeclareMathOperator{\injec}{inj}

\newcommand{\M}{\mathcal{M}}

\newcommand{\N}{\mathcal{N}}

\newcommand{\R}{\mbb{R}}

\newcommand{\ra}{\right\rangle}
\newcommand{\la}{\left\langle}
\newcommand{\e}{\varepsilon}

\newcommand{\pd}{\partial}
\newcommand{\cd}{\nabla}

\DeclareMathOperator{\Vol}{Vol}
\DeclareMathOperator{\G}{G}
\DeclareMathOperator{\Id}{Id}

\DeclareMathOperator{\Rm}{Rm}
\DeclareMathOperator{\Rc}{Rc}
\DeclareMathOperator{\Sc}{Sc}

\protected\def\vts{%
  \ifmmode
    \mskip0.5\thinmuskip
  \else
    \ifhmode
      \kern0.08334em
    \fi
  \fi
}

\newcommand{\lb}{\left(}
\newcommand{\rb}{\right)}
\newcommand{\lsb}{\left[}
\newcommand{\rsb}{\right]}

\def\labelitemi{--}
\def\ba #1\ea {\begin{align} #1\end{align}}
\def\bann #1\eann {\begin{align*} #1\end{align*}}
\def\ben #1\een {\begin{enumerate} #1\end{enumerate}}
\def\bi #1\ei {\begin{itemize}\renewcommand\labelitemi{--} #1\end{itemize}}

\newcommand{\inner}[2]{\left\langle#1,#2\right\rangle} %inner product 

\address{Department of Mathematics, University of Tennessee Knoxville, Knoxville TN, USA, 37996-1320}
\address{School of Mathematical and Physical Sciences, The University of Newcastle, Newcastle, NSW, Australia, 2308}
\email{mlangford@utk.edu, mathew.langford@newcastle.edu.au}

\address{School of Mathematical Sciences, Queen Mary University of London, Mile End Road, London, UK, E1 4NS}
\email{h.nguyen@qmul.ac.uk}

\begin{document}
\title[Quadratically pinched hypersurfaces of the sphere]{Quadratically pinched hypersurfaces of the sphere via mean curvature flow with surgery}
\author{Mat Langford}
\author{Huy The Nguyen}
\address{}
\email{}
%; Secondary 35K55, 58J35, 53C21}
\begin{abstract} 
We study mean curvature flow in $\mathbb S_K^{n+1}$, the round sphere of sectional curvature $K>0$, under the quadratic curvature pinching condition $|A|^{2} < \frac{1}{n-2} H^{2} + 4 K$ when $n\ge 4$ and $|A|^{2} < \frac{3}{5}H^{2}+\frac{8}{3}K$ when $n=3$. This condition is related to a famous theorem of Simons \cite{Si68}, which states that the only \emph{minimal} hypersurfaces satisfying $\vert A\vert^2<nK$ are the totally geodesic hyperspheres. It is related to but distinct from the ``two-convexity'' condition studied in \cite{BrHu17,HuSi09}. Notably, in contrast to two-convexity, it allows the mean curvature to change sign. We show that the pinching condition is preserved by mean curvature flow, and obtain a ``cylindrical'' estimate and corresponding pointwise derivative estimates for the curvature. %These estimates do not depend on the maximal time of existence of the flow.% (unlike analogous estimates obtained when the ambient space is Euclidean \cite{HuSi09}. See also \cite{BrHu17}).
As a result, we find that the flow becomes either uniformly convex or quantitatively cylindrical in regions of high curvature. This allows us to apply the surgery apparatus developed by Huisken and Sinestrari \cite{HuSi09} (cf. \cite{HK2}). We conclude that any smoothly, properly, isometrically immersed hypersurface $\M$ of $\mathbb S_K^{n+1}$ satisfying the pinching condition is diffeomorphic to $\mathbb S^n$ or the connected sum of a finite number of copies of $\mathbb S^1\times \mathbb S^{n-1}$. If $\M$ is embedded, then it bounds a 1-handlebody. The results are sharp when $n\ge 4$.
\end{abstract}
\maketitle

\tableofcontents

\section{Introduction}

Beginning in the late 1960's, Simons \cite{Si68} and others \cite{AlencarDoCarmo,ChengNakagawa,ChernDoCarmoKobayashi,Okumura} obtained rigidity theorems for minimal and constant mean curvature hypersurfaces in the sphere under certain bounds on the second fundamental form (depending on the dimension and the value of the mean curvature). The results are obtained by exploiting \emph{Simons' identity}, a Bochner-like formula which relates the Hessian of the mean curvature (which vanishes for a constant mean curvature hypersurface) to the Laplacian of the second fundamental form. Simons' theorem, for example, states that the only minimal hypersurfaces of $\mathbb S^{n+1}$ satisfying $\vert A\vert^2\le n$, where $A$ denotes the second fundamental form, are the totally geodesic hyperspheres (which satisfy $\vert A\vert^2\equiv 0$) and the Clifford hypersurfaces (which satisfy $\vert A\vert^2=n$).

Such results can be improved upon using curvature flows, which remove the constant mean curvature restriction. Indeed, Huisken \cite{Hu87} showed that, under mean curvature flow, hypersurfaces of the sphere $\mathbb{S}_K^{n+1}$ of sectional curvature\footnote{We find it convenient to work without normalizing the curvature $K$, as it serves as a natural scale parameter.} $K$, $n\ge 2$, satisfying the quadratic curvature pinching condition
\begin{equation}\label{eq:Huisken pinching}
\begin{cases}
\vert A\vert^2<\frac{1}{n-1}H^2+2K&\text{if}\;\; n\geq 3 \vts , \smallskip \\
\vert A\vert^2<\frac{3}{4}H^2+\frac{4}{3}K&\text{if}\;\; n=2
\end{cases}
\end{equation}
shrink, preserving the inequality, either to a ``round'' point in finite time or to a totally geodesic hypersphere in infinite time. In case $n\geq 3$, this behaviour is sharp in the sense that there exist hypersurfaces of the form $\mathbb S^1(r)\times \mathbb S^{n-1}(s)$, $r^2+s^2=1$, on which $\vert A\vert^2-\frac{1}{n-1}H^2$ can be made arbitrarily close to $2$. Andrews \cite{An02} obtained a sharper result when $n=2$: he showed that, under a different (fully nonlinear) curvature flow, positive sectional curvature (which is equivalent to the inequality $\vert A\vert^2<H^2+2$) is preserved, and solutions converge either to round points in finite time, or totally geodesic spheres in infinite time.

We will develop these results further by allowing a weaker curvature pinching condition. Namely, we study, for $n\ge 3$, hypersurfaces of $\mathbb S_K^{n+1}$ satisfying
\begin{equation}\label{eq:strict quadratic pinching}
\begin{cases}
\vert A\vert^2<\frac{1}{n-2}H^2+4K &\text{if}\;\; n\geq 4 \vts , \smallskip \\
\vert A\vert^2<\frac{3}{5}H^2+\frac{8}{3}K &\text{if}\;\; n=3\,.
\end{cases}
\end{equation}
The analysis is much more complicated under the weaker condition \eqref{eq:strict quadratic pinching}, since we can no longer expect solutions to shrink to a round point at a finite time singularity --- further singularities and topologies are possible. The purpose of the pinching condition \eqref{eq:strict quadratic pinching} is to ensure that the only additional singularities are (possibly degenerate) ``neck-pinch'' singularities. In this respect, our results are sharp (when $n\ge 4$), since there exist hypersurfaces of $\mathbb S^{n+1}$ the form $\mathbb S^2(r)\times \mathbb S^{n-2}(s)$, $r^2+s^2=1$, on which $\vert A\vert^2-\frac{1}{n-2}H^2$ can be made arbitrarily close to $4$. Once this is established, we are able to make use of the robust surgery construction of Huisken and Sinestrari \cite{HuSi09}, which allows us to replace the singular neck regions by almost spherical caps, and thereby continue the flow (cf. \cite{BrHu17}). Since the estimates hold in the presence of surgeries, with constants that do not depend on the maximal time of existence, we find, after a finite number of surgeries, that the initial hypersurface has decomposed into a finite number of components, each of which is either a ``small'' $\mathbb S^n$, the Cartesian product of $\mathbb S^1$ with a ``small'' $\mathbb S^{n-1}$, or a ``large'' $\mathbb S^n$. As a consequence, we obtain a classification of diffeomorphism types for hypersurfaces satisfying the pinching condition.

\begin{theorem}\label{thm:main theorem}
Every properly, isometrically immersed hypersurface $X:\M\to\mathbb S^{n+1}_K$ of $\mathbb S^{n+1}_K$ satisfying \eqref{eq:strict quadratic pinching} is diffeomorphic either to $\mathbb  S^n$ or to a connected sum of finitely many copies of $\mathbb S^1\times \mathbb S^{n-1}$. Indeed, there exists a 1-handlebody $\Omega$ and an immersion $\overline X:\Omega\to \mathbb S_K^{n+1}$ such that $\pd\Omega$ is diffeomorphic to $\M$ and $\overline X|_{\pd\Omega}=X$. If $X$ is an embedding, then so is $\overline X$.
\end{theorem}

Our arguments follow those of Huisken and Sinestrari \cite{HuSi09} --- however we first establish a cylindrical estimate (Theorem \ref{thm:cylindrical estimate}) by way of Stampacchia iteration, and then use this to obtain pointwise derivative estimates for the curvature using the maximum principle (Theorems \ref{thm:gradient estimate} and \ref{thm:Hessian estimate}% and \ref{thm:higher order estimates}
). These estimates allow us to %establish the crucial neck detection and neck continuation lemmas, which allow us to 
apply, virtually unmodified, the Huisken--Sinestrari surgery algorithm after pulling the flow locally up to the tangent space to $\mathbb S_K^{n+1}$. In particular, we do not use positive mean curvature and convexity estimates follow as a consequence of the cylindrical estimates.

The cylindrical estimate may be viewed as a partial generalization of Simons' famous theorem mentioned above. %Indeed, for minimal hypersurfaces, it immediately implies Simons' theorem when $n=4$. 
It also implies a new rigidity theorem for ancient solutions to mean curvature flow in the sphere (Corollary \ref{cor:rigidity ancient}). The key to proving it is a Poincar\'e-type inequality for $W^{2,2}$-functions supported away from ``cylindrical'' points of the hypersurface (Proposition \ref{prop:Poincare}).

We also obtain novel noncollapsing estimates (Corollary \ref{cor:noncollapsing}) which are preserved by the surgery algorithm (see \S \ref{ssec:pinching under surgery}). These are not actually required for the proof of Theorem \ref{thm:main theorem}, but we include them here as they may be of some use in obtaining further applications (cf. \cite{BuzanoHaslhoferHershkovitsTori,BuzanoHaslhoferHershkovitsSphere,Mramor}).

The work of Huisken and Sinestrari was generalized in a different direction by Brendle and Huisken \cite{BrHu17} who, building upon earlier work of Andrews \cite{AndrewsRiemannian}, studied the evolution (with surgeries) by a certain fully nonlinear flow of hypersurfaces in compact Riemannian ambient spaces satisfying the two-convexity condition
\[
\lambda_1+\lambda_2>2\sqrt{-K}\,,
\]
where $\lambda_1\le \lambda_2\le\dots\le\lambda_n$ are the principal curvatures of the hypersurface and $K\le 0$ is a lower bound for the sectional curvatures of the ambient space. For hypersurfaces of the sphere, this becomes ordinary two-convexity,
\[
\lambda_1+\lambda_2>0\,.
\]
A consequence of their work is a version of Theorem \ref{thm:main theorem} for hypersurfaces with the quadratic pinching condition \eqref{eq:strict quadratic pinching} replaced by two-convexity. Note that two-convexity is neither stronger nor weaker than the quadratic pinching condition \eqref{eq:strict quadratic pinching}. Indeed, unlike two-convexity, the quadratic condition \eqref{eq:strict quadratic pinching} is invariant under orientation reversal, and therefore allows the mean curvature to change sign. Moreover, the Brendle--Huisken approach holds only for \emph{embedded} hypersurfaces, since they require noncollapsing estimates (see \cite{An12,ALM13,AndrewsHanLiWei}) to obtain a gradient estimate for the curvature (cf. \cite{HK1}).

\section*{Acknowledgements}  
M.~Langford was supported by an Alexander von Humboldt fellowship and an Australian Research Council DECRA fellowship. 

H.~T.~Nguyen was supported by the EPSRC grant EP/S012907/1.

\section{Preliminaries}

\subsection{Hypersurfaces of $\mathbb{S}^{n+1}_K$}

Here we recall the fundamental identities for immersed hypersurfaces $X:\M\to\mathbb{S}^{n+1}_K$ of the sphere $\mathbb{S}^{n+1}_K$ of sectional curvature $K>0$. First, recall the Gauss equation
\begin{align}\label{eq:spatial Gauss}
\Rm_{ijkl}&=A_{ik}A_{jl}-A_{il}A_{jk}+K({g}_{ik}{g}_{jl}-{g}_{il}{g}_{jk})\,,
\end{align}
where $A^2_{ij}:=A_i{}^pA_{pj}$, and its traces
\begin{align*}
\Rc_{ik}&=(HA_{ik}-A^2_{ik})+(n-1)K{g}_{ik}
\end{align*}
and
\begin{align*}
\Sc&=H^{2}-|A|^{2}+n(n-1)K\,.
\end{align*}
 
The Codazzi equation,
\[
\nabla_{k}A_{ij}=\nabla_iA_{kj}\,,
\]
implies that the covariant differential $\nabla A$ of the second fundamental form is totally symmetric.

Combining the Gauss and Codazzi equations yields Simons' identity
\begin{align}\label{eq:Simons' identity}
\nabla_{\!(i}\nabla_{j)}A_{kl}-\nabla_{\!(k}\nabla_{l)}A_{ij} = A_{ij}A^2_{kl}-A_{kl}A^2_{ij}+K(g_{ij}A_{kl}-g_{kl}A_{ij}),
\end{align}
where brackets indicate symmetrization about the enclosed components, and its trace
\begin{align}\label{eq:contracted Simons' identity}
\Delta A_{ij}=\nabla_{i}\nabla_{j}H+HA^2_{ij}-|A|^{2}A_{ij}-K(Hg_{ij}-nA_{ij})\,.
\end{align}

By splitting $\nabla A$ into is trace and trace-free parts, we obtain the Kato inequality
\begin{align}\label{ineqn_Kato}
|\nabla A|^{2} \geq \frac{3}{n+2} | \nabla H|^2.
\end{align}

\subsection{Mean curvature flow in $\mathbb{S}_K^{n+1}$}

%In this section we gather together all the evolution equations that we will be using in this paper. Each of the equations is derived in \cite{Hu87}.

Next, we recall the fundamental identities for a family of hypersurfaces $X:\M\times I\to\mathbb{S}_K^{n+1}$ evolving by mean curvature flow. We make use of the \emph{time-dependent connection} of Andrews and Baker \cite{AnBa10}, which differentiates time-dependent tangent vector fields $V$ on $\M$ in space-time directions $\xi\in T(\M\times I)$ in the obvious way:
\bann
dX(\nabla_\xi V):={}& \big(D_\xi \big[dX(V)\big]\big)^\top,%\nonumber\\
%={}&dX([\pd_t,V]-HA(V))\,.
\eann
where $D$ is the pullback to $\M\times I$ of the ambient connection, $dX$ is the differential of $X$ and $(\,\cdot\,)^\top$ the projection onto $dX(T\M)$. Observe that
\ba\label{eq:spacetime connection}
\nabla_\xi V={}&[\pd_t,V]-HA(V)\,,
\ea
where $[\,\cdot\,,\cdot\,]$ denotes the Lie bracket and we conflate the second fundamental form with the Weingarten map. 

Note that $\nabla_\xi$ agrees with the Levi-Civita covariant derivative on the \emph{spatial tangent bundle} $\{\xi\in T(\M\times I):dt(\xi)=0\}$ (which we conflate with $T\M$) when $\xi$ has no $\partial_t$ component, where $\partial_t$ is the canonical tangent vector field to $I$. The main advantage of working with the time-dependent connection (as opposed to the Lie derivative) is that the induced metric tensor $g$ is $\nabla_t$-parallel:
\[
\nabla_{t}g= 0\,,
\]
where $\cd_t:=\cd_{\pd_t}$.

Denote by ${}^X\overline\Rm$ the curvature tensor of the pullback connection ${}^XD$ and let $U$ and $V$ be any pair of time-dependent tangent vector fields. Observe, on the one hand, that
\ba\label{eq:pullback identities}
{}^X\overline{\Rm}(\pd_t,U)V={}&dX\lsb\Rm(\pd_t,U)V+A(U,V)\cd H-\cd_VHA(U)\rsb\nonumber\\
{}&+\lsb\cd_tA(U,V)-\cd_U\cd_VH-HA^2(U,V)\rsb\nu\,,
\ea
where we conflate $\cd H$ with the gradient of $H$. On the other hand,
\bann
{}^X\overline{\Rm}(\pd_t,U)V={}&\overline{\Rm}(dX(\pd_t),dX(U))(dX(V))\\
={}&-H\overline{\Rm}(\nu,dX(U))dX(V)\\
={}&HKg(U,V)\nu\,.
\eann
Resolving \eqref{eq:pullback identities} into tangential and normal components, we obtain the ``temporal'' Gauss-Codazzi equations
\ba\label{eq:temporal Gauss}
\Rm(\pd_t,U)V=\cd_VHA(U)-A(U,V)\cd H
\ea
and
\ba\label{eq:temporal Codazzi}
\cd_tA=\cd^2H+HA^2+KHg\,,
\ea
respectively, where in \eqref{eq:temporal Codazzi} both sides are understood as tensors on the spatial tangent bundle.

Combining the Codazzi identity \eqref{eq:temporal Codazzi} with the contracted Simons identity \eqref{eq:contracted Simons' identity} yields an evolution equation for the second fundamental form:
\begin{align}
(\nabla_{t}-\Delta)A={}&(|A|^{2}+nK)A-2nK(A-\tfrac{1}{n}Hg)\,,\label{eq_sff}
\end{align}
where $\Delta$ is the spatial Laplacian. Tracing yields
\begin{align}
(\partial_{t}-\Delta)H={}&(|A|^{2}+nK)H\,,\label{eq_H}
\end{align}
which immediately yields
\begin{align}
(\partial_{t}-\Delta)H^{2}={}&-2|\nabla H|^{2}+2(|A|^{2}+nK)H^{2}\,.\label{evol_Hsquare}
\end{align}

Since $g$ is $\nabla_t$-parallel, \eqref{eq_sff} immediately yields
\begin{align}
(\partial_{t}-\Delta)|A|^{2}={}&-2|\nabla A|^{2}+2(|A|^{2}+nK)|A|^{2}\label{evol_Asquare}\\
{}&-4nK(|A|^{2}-\tfrac{1}{n}H^{2}),\nonumber
\end{align}
where $\nabla A$ is the spatial covariant differential of $A$. 

Given tensor fields $S$ and $T$, we denote by $S\ast T$ any tensor field resulting from linear combinations of metric contractions of $S\otimes T$.
By \eqref{eq:spacetime connection} and \eqref{eq:temporal Gauss},
\bann
\cd_t(\cd T)={}&\cd(\cd_tT)+A\ast A\ast \cd T+A\ast\cd A\ast T.
\eann
By \eqref{eq:spatial Gauss},
\bann
\Delta(\cd T)={}&\cd(\Delta T)+A\ast A\ast \cd T+K\ast \cd T+A\ast \cd A\ast T\,.
\eann
Thus,
\bann
(\cd_t-\Delta )(\cd A)={}&\cd\lsb(\cd_t-\Delta)A\rsb+A\ast A\ast \cd A+K\ast\cd A\\
={}&A\ast A\ast \cd A+K\ast\cd A\,,
\eann
and hence, by Young's inequality,
\begin{align}
(\partial_t-\Delta)|\nabla A|^2\le{}& -2 |\nabla^2 A|^2 + c_n(|A|^2+nK)|\nabla A|^2\,,\label{eqn_evolderiv}
\end{align}
where $c_n$ is a constant that depends only on $n$.

Similarly,
\bann
(\cd_t-\Delta )(\cd^2 A)={}&A\ast A\ast \cd^2 A+A\ast\cd A\ast \cd A+K\ast\cd^2 A\,,
\eann
and hence
\begin{align}
(\partial_t-\Delta)|\nabla^2\!A|^2\le{}&\!\!-2|\nabla^3\!A|^2\! + c_n\!\lsb (|A|^2\!+nK)|\nabla^2A|^2\!+|\cd A|^4\rsb,\label{eqn_evolderiv2}
\end{align}
where $c_n$ is a constant that depends only on $n$.

Similar inequalities hold for higher derivatives of $A$ since, by a straightforward induction argument,
\ba\label{eq:evolve derivatives of A}
(\cd_t-\Delta )(\cd^m A)={}&K\ast\cd^mA+\sum_{i+j+k=m}\cd^iA\ast \cd^jA\ast \cd^kA\,.
\ea

The following ``Bernstein estimates'' are a standard application of the ``rough'' evolution equations \eqref{eq:evolve derivatives of A}. For a proof in the Euclidean case (which carries over with minor modifications) see, for example, {\cite[Theorem 6.24]{EGF}}.

\begin{proposition}[Bernstein estimates]\label{prop:Bernstein}
Let $X:\M\times [0,\lambda K^{-1}]\to\mathbb{S}_K^{n+1}$ be a solution to mean curvature flow. If
\[
\max_{\M\times[0,\lambda K^{-1}]}\vert A\vert^2\le \Lambda_0K\,,
\]
then
\[
t^m\vert\cd^m A\vert^2\le \Lambda_{m}K\,,
\]
where $\Lambda_{m}$ depends only on $n$, $m$, $\lambda$ and $\Lambda_0$.
\end{proposition}

\subsection{A Poincar\'e-type inequality}

The following Poincar\'e-type inequality (cf. \cite[5.4 Lemma]{Hu84}) is crucial to obtaining the cylindrical estimate (via Stampacchia iteration) in Section \ref{ssec:cylindrical estimate}.
\begin{proposition}\label{prop:Poincare}
Given $n\geq 3$, $\alpha\in(0,1)$ and $\eta\in(0,\frac{1}{n-2+\alpha}-\frac{1}{n-1})$ there exists $\gamma=\gamma(n,\alpha,\eta)>0$ with the following property: Let $X:\mc M^n\to \mbb S^{n+1}_K$ be a smoothly immersed hypersurface and let $u\in W^{1,2}(\mc M)$ be a function satisfying $\operatorname{spt}u\subset U_{\alpha,\eta}$, where, introducing the functions 
\[
f_{1,\eta}:=|A|^2-\left(\frac{1}{n-1}+\eta\right)H^2
\]
and
\[
g_{2,\alpha}:=|A|^2-\frac{1}{n-2+\alpha}H^2-2(2-\alpha)K\,,
\]
the ``acylindrical'' set $U_{\alpha,\eta}\subset \mathcal{M}$ is defined by
\[
U_{\alpha,\eta}:=\{x\in \mathcal{M}:f_{1,\eta}(x)\geq 0\geq g_{2,\alpha}(x)\}\,.
\]
For any $r\geq 1$,
\[
\gamma\int u^2W\,d\mu\leq \int u^2\left(r^{-1}\frac{|\nabla u|^2}{u^2}+r\frac{|\nabla A|^2}{W}+K\right)d\mu,
\]
where%, setting $a:=\frac{1}{n-2+\alpha}-\left(\frac{1}{n-1}+\eta\right)$ and $b:=2(2-\alpha)K$,
%\[
%W^2:=aH^2+b\,.
%\]
\[
W:=\left(\frac{1}{n-2+\alpha}-\frac{1}{n-1}-\eta+\frac{\alpha}{2n(n-1)}\right)H^2+2(2-\alpha)K\,.
\]
\end{proposition}
\begin{proof}
By a straightforward scaling argument, it suffices to prove the claim when $K=1$. By a standard approximation argument, we may assume that $u$ is smooth. 

Recall Simons' identity
\[
\nabla_{(i}\nabla_{j)}A_{kl}-\nabla_{(k}\nabla_{l)}A_{ij}=\mathrm{C}_{ijkl}\,,
\]
where the brackets denote symmetrization and
\[
\mathrm{C}:=A\otimes A^2-A^2\otimes A+(g\otimes A-A\otimes g)\,.
\]
We claim that
\begin{equation}\label{eq:Poincare1}
\gamma\, W^3\leq |\mathrm{C}|^2+1\quad\text{in}\quad U_{\alpha,\eta}
\end{equation}
on any immersed hypersurface %\footnote{In fact, we prove that the inequality holds in a stronger algebraic sense.} 
$X:\mathcal{M}^n\to \mathbb{S}^{n+1}$ for some positive $\gamma=\gamma(n,\alpha,\eta)$. Indeed, if this is not the case, then there is a sequence $\{\vec \lambda^k\}_{k\in\mathbb N}$ of vectors $\vec\lambda^k\in \R^n$ (corresponding to principal curvatures of a sequence of hypersurfaces) satisfying
\[
f_{1,\eta}(\vec\lambda^k):=|\vec\lambda^k|^2-\frac{1}{n-1}\tr(\vec\lambda^k)^2-\eta \tr(\vec\lambda^k)^2\geq 0
\]
and
\[
g_{2,k}(\vec\lambda^k):=|\vec\lambda^k|^2-\frac{1}{n-2+k}\tr(\vec\lambda^k)^2-2(2-k)\leq 0\,,
\]
where $\tr(\vec\lambda):=\sum_{i=1}^n\lambda_i$, but
\[
\frac{|\mathrm{C}(\vec\lambda^k)|^2+1}{W^3(\vec\lambda^k)}\to 0
\]
as $k\to\infty$, where
\[
|\mathrm{C}(\vec\lambda)|^2:=\sum_{i,j=1}^n(\lambda_j-\lambda_i)^2(\lambda_i\lambda_j+1)^2
\]
and
\[
W(\vec\lambda):=\left(\frac{1}{n-2+\alpha}-\frac{1}{n-1}-\eta+\frac{\alpha}{2n(n-1)}\right)\tr(\vec\lambda)^2+2(2-\alpha)\,.
\]
Set $r^2_{k}:=W(\vec\lambda^k)^{-1}\to 0$ and $\hat\lambda^k:=r_k\vec\lambda^k$. Observe that
\[
\vert \hat\lambda^k\vert\leq \frac{2n(n-1)}{\alpha(n-2)}=:c(n,\alpha)%1+\frac{\frac{1}{n-1}+\eta}{\frac{1}{n-2+\alpha}-\frac{1}{n-1}-\eta}=:c(n,\alpha,\eta)
\]
and hence, up to a subsequence, $\hat\lambda^k\to \hat\lambda\in \R^n$. Computing
%\[
%1=r^2_k W^2(\vec\lambda^k)= aH^2(\hat\lambda^k)+r^2_k b\,,
%\]
\[
\left(|\hat \lambda^k|^2-\frac{1}{n-1}\tr(\hat \lambda^k)^2\right)-\eta \tr(\hat\lambda^k)^2=r^2_k f_{1,\eta}(\vec\lambda^k)\geq0
\]
and
\[
\left(|\hat \lambda^k|^2-\frac{1}{n-2+\alpha}\tr(\hat \lambda^k)^2\right)-2(2-\alpha)r_k^2=r^2_k g_{2,\alpha}(\vec\lambda^k)\leq 0\,,
\]
we find
%\begin{equation}\label{eq:Poincarecontra1}
%H^2(\hat\lambda)=a^{-1}>0\,,
%\end{equation}
\begin{equation}\label{eq:Poincarecontra2}
\left(|\hat \lambda|^2-\frac{1}{n-1}\tr(\hat \lambda)^2\right)\geq \eta \tr(\hat \lambda)^2
\end{equation}
and
\begin{equation}\label{eq:Poincarecontra3}
\left(|\hat \lambda|^2-\frac{1}{n-2+\alpha}\tr(\hat \lambda)^2\right)\leq 0\,.
\end{equation}
On the other hand,
\[
\sum_{i,j=1}^n\left(\hat\lambda^k_i\hat\lambda^k_j(\hat\lambda^k_j-\hat\lambda^k_i)\right)^2 +2r_k^2\hat\lambda^k_i\hat\lambda^k_j(\hat\lambda^k_j-\hat\lambda^k_i)^2 +r_k^4(\hat\lambda^k_j-\hat\lambda^k_i)^2 = r_k^6|\mathrm{C}(\vec\lambda^k)|^2
\]
so that
\begin{equation}\label{eq:Poincarecontra4}
\sum_{i,j=1}^n\left(\hat\lambda_i\hat\lambda_j(\hat\lambda_j-\hat\lambda_i)\right)^2=0\,.
\end{equation}
Together, %\eqref{eq:Poincarecontra1}, 
\eqref{eq:Poincarecontra2}, \eqref{eq:Poincarecontra3} and \eqref{eq:Poincarecontra4} are in contradiction: \eqref{eq:Poincarecontra4} implies that $\hat\lambda$ has a null component of multiplicity $m$ and a non-zero component, $\kappa$ say, of multiplicity $n-m$. %Since \eqref{eq:Poincarecontra1} implies that $H^2(\hat\lambda)>0$, 
The inequalities \eqref{eq:Poincarecontra2} and \eqref{eq:Poincarecontra3} then yield
\[
\left(n-m-\frac{(n-m)^2}{n-1}\right)\kappa^2\geq \eta(n-m)^2\kappa^2>0
\]
and
\[
\left(n-m-\frac{(n-m)^2}{n-2+\alpha}\right)\kappa^2\leq 0\,.
\]
which together imply that $m\in(1,2-\alpha]$, which is impossible. This proves \eqref{eq:Poincare1}. 

Using \eqref{eq:Poincare1}, we can estimate
\begin{comment}
\begin{align*}
\gamma\!\int\! u^2Wd\mu\leq{}& \int  \frac{u^2}{W^{2}}\left(|\mathrm{C}|^2+1\right)d\mu\\
={}& \int \frac{u^2}{W^{2}}\mathrm{C}^{ijkl}\left(\nabla_i\nabla_jA_{kl}-\nabla_k\nabla_lA_{ij}\right)d\mu+\int  \frac{u^2}{W^{2}}\,d\mu\\
={}&-\int\!\frac{u^2}{W^{2}}\left(2\mathrm{C}^{ijkl}\frac{\nabla_iu}{u}-2\mathrm{C}^{ijkl}\frac{\nabla_iW}{W}+\nabla_i\mathrm{C}^{ijkl}\right)\!\nabla_jA_{kl}\,d\mu\\
{}&+\int\!\frac{u^2}{W^{2}}\left(2\mathrm{C}^{ijkl}\frac{\nabla_ku}{u}-2\mathrm{C}^{ijkl}\frac{\nabla_kW}{W}+\nabla_k\mathrm{C}^{ijkl}\right)\!\nabla_lA_{ij}\,d\mu\\
{}&+\int  \frac{u^2}{W^{2}}\,d\mu\\
\leq{}&C\left[\frac{u^2}{W^{2}}\left(W^{\frac{3}{2}}\frac{|\nabla u|}{u}+W^{\frac{1}{2}}|\nabla W|+W|\nabla A|\right)|\nabla A|\,d\mu\right.\\
{}&\qquad\left.+\int u^2\,d\mu\right]\\
\leq{}&C\left[\int u^2\left(\frac{|\nabla u|}{u}+\frac{|\nabla A|}{W^{\frac{1}{2}}}\right)\frac{|\nabla A|}{W^{\frac{1}{2}}}\,d\mu+\int u^2\,d\mu\right],
\end{align*}
\end{comment}
\begin{align*}
\gamma\!\int\! u^2Wd\mu\leq{}& \int  \frac{u^2}{W^{2}}\left(|\mathrm{C}|^2+1\right)d\mu\\
={}& \int \frac{u^2}{W^{2}}\left(\mathrm{C}\ast\nabla^2A+1\right)\,d\mu\\
={}&\int\!\frac{u^2}{W^{2}}\!\left(\frac{\nabla u}{u}\ast\mathrm{C}+\frac{\nabla W}{W}\ast\mathrm{C}+\nabla\mathrm{C}\right)\!\ast\nabla A\,d\mu+\int  \frac{u^2}{W^{2}}\,d\mu\\
\leq{}&C\left[\int\frac{u^2}{W^{2}}\left(W^{\frac{3}{2}}\frac{|\nabla u|}{u}+W^{\frac{1}{2}}|\nabla W|+W|\nabla A|\right)|\nabla A|\,d\mu\right.\\
{}&\qquad\left.+\int u^2\,d\mu\right]\\
\leq{}&C\left[\int u^2\left(\frac{|\nabla u|}{u}+\frac{|\nabla A|}{W^{\frac{1}{2}}}\right)\frac{|\nabla A|}{W^{\frac{1}{2}}}\,d\mu+\int u^2\,d\mu\right],
\end{align*}
where $C$ denotes any constant which depends only on $n$, $\alpha$ and $\eta$. The claim now follows from Young's inequality.
\end{proof}

\subsection{The Sobolev inequality}

We shall also require the following Sobolev inequality for the Stampacchia iteration argument in Section \ref{ssec:cylindrical estimate}. It may be obtained from \cite[Theorem 2.1]{HoSp74} (cf. \cite{MiSi73}) by the substitution $u\mapsto u^2$.
\begin{theorem}\label{thm:Sobolev}
Let $X:\M\to \mathbb{S}_K^{n+1}$, $n\ge 3$, be a hypersurface of the sphere of sectional curvature $K>0$ and let $u:\M\to \R$ a $W^{1,2}$ be function. If
\[
\mu(\mathrm{spt}(u))\leq \frac{\omega_n}{n+1}K^{-\frac{n}{2}}\,,
\]
then
\begin{equation}\label{eq:Sobolev}
\left(\int u^{2^\ast}d\mu\right)^{\frac{1}{2^\ast}}\leq c_n\int\left(|\nabla u|^2+u^2H^2\right)d\mu\,,
\end{equation}
where $q^\ast:=\frac{nq}{n-q}$, and $c_n$ is a constant that depends only on $n$.
\end{theorem}

\section{Preserved curvature conditions}

\subsection{Quadratic curvature condition}\label{ssec:quadratic pinching}
If the strict quadratic curvature inequality \eqref{eq:strict quadratic pinching} holds on a hypersurface of $\mathbb{S}_K^{n+1}$, $n\ge 3$, then we can find some $\alpha>0$ such that
\begin{align}\label{eq:uniform quadratic pinching}
|A|^{2}\leq\frac{1}{n-2+\alpha}H^{2}+2(2-\alpha)K\,.  
\end{align}
This inequality is preserved under mean curvature flow when $n\ge 3$ (note that $\alpha>\frac{2}{3}$ when $n=3$).

%Having fixed such an $\alpha$, it will be useful to define the constants $\beta_{n}=2(2-\alpha)$ and $\alpha_{n}=\frac{1}{n-2+\alpha}$.

\begin{proposition}[Cf. {\cite[1.4 Lemma]{Hu86}}]
Let $X:\mc M^n\times[0,T)\to \mbb S_K^{n+1}$, $n\ge 3$, be a solution to mean curvature flow such that \eqref{eq:uniform quadratic pinching} holds on $\mc M^n\times\{0\}$ for some $\alpha\in(0,1)$. If $n\ge 4$, or if $n=3$ and $\alpha\ge \frac{2}{3}$, then \eqref{eq:uniform quadratic pinching} holds on $\mc M^n\times\{t\}$ for all $t\in[0,T)$.
\end{proposition}
\begin{proof}
Suppose that \eqref{eq:uniform quadratic pinching} holds on $\M\times\{0\}$ for some $\alpha\in(0,1)$ when $n\ge 4$ or some $\alpha\in(\frac{2}{3},1)$ when $n=3$. Setting
\[
a_{n}:=\frac{1}{n-2+\alpha}\;\; \text{and}\;\; b_{n}:=2(2-\alpha)\,,
\]
we compute, using \eqref{evol_Hsquare} and \eqref{evol_Asquare},
\begin{align*}
(\partial_{t}-\Delta)\!\left(|A|^{2}-a_{n} H^{2}\right)={}&- 2 \left(|\nabla A|^{2} 
- a_{n}|\nabla H|^{2}\right) + 2 b_{n} K ( |A|^{2} + nK) \\
&+ 2 ( |A|^{2} - a_{n} H^{2} - b_{n} K ) ( | A|^{2} + n K )\\
{}& - 4 nK \left(|A|^{2} - \tfrac{1}{n}H^{2}\right).
\end{align*}
Since $\frac{2}{2n-b_{n}}=a_n$ and $\frac{n}{2n-b_{n}}\leq 1$, we can estimate
\begin{align*}
2b_{n} K (|A|^{2} + nK)-4 nK&\left(|A|^{2}- \tfrac{1}{n} H^{2} \right)\\
={}& 2 K\left((b_{n}- 2n)|A|^{2} + 2 H^{2} + b_{n} n K\right) \\
={}&-2K( 2n-b_{n})\left(|A|^{2} - \tfrac{ 2}{2n - b_{n}} H^{2} - \tfrac{ nb_{n} }{ 2n - b_{n}} K\right)\\
\leq{}&-2K(2n-b_{n})\left(|A|^{2}-a_n H^{2}-b_{n}K\right)\,.
\end{align*}
Estimating $\frac{2}{2n-b_{n}} \leq \frac{3}{n+2}$ and applying \eqref{ineqn_Kato}, we arrive at
\begin{align*}
(\partial_{t}-\Delta)\big(|A|^{2}-a_{n} H^{2}-{}&b_nK\big)\\
\leq{}& 2\left(|A|^{2}+(b_n-n)K\right)\left(|A|^{2}-a_{n} H^{2}-b_{n}K\right).
\end{align*}
The claim now follows from the maximum principle.
\end{proof}

\subsection{Rigidity}

The quadratic curvature condition \eqref{eq:strict quadratic pinching} is optimal for cylindrical estimates and connected sum theorems in dimensions $n\geq 4$. Indeed, consider the hypersurfaces $M^{k, n-k}(r,s) = \mathbb S^{k}(r) \times \mathbb S^{n-k}(s)$, $r^{2} + s^{2} = 1$, of $\mathbb S^{n+1}$, where $\mathbb S^k(r)$ is the $k$ dimensional sphere of radius $r$. The second fundamental forms have eigenvalues $ \lambda$, with multiplicity $k$, and $\mu$, with multiplicity $n-k$, such that $ \lambda \mu = -1$. 

\begin{comment}
Consider first the case $k=1$. Without loss of generality, we may assume that $\lambda = -\frac{s}{r}$ and $\mu =\frac{r}{s}$. In that case,
\begin{align*}
|A|^{2} &= \frac{( n -1)r^{4} + s^{4}}{r^{2}s^{2}}
\end{align*}           
and
\begin{align*}
H &= \frac{(n-1) r^{2} - s^{2} }{ rs},
\end{align*}           
so that
\begin{align*}
H^{2} &= \frac{(n-1)^2r^{4}-2(n-1)r^2s^2+s^{4}}{ r^{2}s^{2}}.   
\end{align*}           
The hypersurface is minimal when $n-1=\frac{s^{2}}{r^{2}}$; i.e. $s^{2} = \frac{n-1}{n}$. As a minimal submanifold, it is stationary under the mean curvature flow. It lies outside of Huisken's curvature condition \eqref{eq:Huisken pinching} because it satisfies 
\begin{align*}
|A|^2 - \frac1{n-1} H^2 - 2 = %\frac{(n-2)}{(n-1)} \frac{ s^2 } { r^2}
n-2 >0.
\end{align*}
On the other hand, %\textcolor{red}{[This is not right: since $\vert A\vert^2=n$, it should be $n-4$].}
\begin{align*}
%|A|^2 - \frac{1}{n-2} H^2 - 4 = \frac{(n-1) \left( 1 - \frac{ n-1}{ n-2} \right) r ^ 4 - 4 r ^ 2 s ^ 2 } { r ^ 2 s ^ 2 } < 0.
|A|^2 - \frac{1}{n-2} H^2 - 4=n-4\,,
\end{align*}
so it lies in the boundary of the condition \eqref{eq:strict quadratic pinching} when $n=4$. %We shall see that the mean curvature flow may converge to such minimal submanifolds in infinite time (but not finite time).

%\textcolor{red}{[If we do converge to a minimal submanifold, it will satisfy $\vert A\vert^2\leq 4\leq n$. Thus, by Simons' theorem, it will be a totally geodesic sphere, unless $n=4$, in which case it could also be a Clifford torus.]}
\end{comment}

Consider the case $k=2$. In this case,
\begin{align*}
 |A|^{2} & = \frac{ 2 s^{4} + ( n-2) r^{4} } { r^{2}s^{2}}
\end{align*}
and
\begin{align*}
 H%&= - 2 \frac {s}{r} + ( n-2) \frac{r}{s} \\
 & = \frac{ ( n-2) r^{2} - 2s ^{2} } { rs},
\end{align*}
so that
\begin{align*}
 H^{2} & = \frac{(n-2)^2r^{4} + 4 s^{4} - 4(n-2) r^{2} s^{2} } { r^{2} s^{2} },
\end{align*}
which then yields% \textcolor{red}{[I think the coefficient should be $\frac{2(n-4)}{n-2}$].}
\begin{align*}
|A|^2-\frac{1}{n-2}H^2-4 = \frac{2(n-4)}{(n-2)} \frac{s^2}{r^2}. 
\end{align*}
Thus, in every dimension $n\ge 4$, we can find, for any $\varepsilon>0$, a hypersurface of the topological type $ \mathbb{S} ^2 \times \mathbb{S}^{ n-2}$ satisfying 
\begin{align*}
|A| ^2 - \frac{1}{ n-2}H^2 - 4 \le \varepsilon\,.
\end{align*} 
So the quadratic bound \eqref{eq:strict quadratic pinching} in Theorem \ref{thm:main theorem} is the best that can be achieved when $n\ge 4$.

%\textcolor{red}{\texttt{Is the right condition positive scalar curvature when $n=3$?}}

\begin{comment}
\subsection{Convexity and quadratic bounds} 
Here we investigate the relationship between quadratic curvature bounds and convexity. 

Observe that
\begin{align}\label{eqn_conv}
|A|^{2} - \frac{ 1}{n-1} H^{2} ={}& -2 \lambda_1 \lambda_2 + \left(\lambda_1 + \lambda_2 - \frac{H}{n-1} \right ) ^2\\
{}& + \sum_{ l=3 } ^{n} \left( \lambda_l - \frac{H}{ n-1} \right) ^ 2   \nonumber
\end{align}
and 
\begin{align}
\label{eqn_two_conv}|A|^{2} - \frac{1}{ n-2} H^{2} ={}& - 2 (\lambda_1 \lambda_2 + \lambda_1\lambda_3 + \lambda_2\lambda_3)\nonumber\\
{}&+ \left(\lambda_1 +\lambda_2+\lambda_3 - \frac{H}{n-2} \right )^2+ \sum_{ l= 4 } ^ n \left ( \lambda _ l - \frac { H}{n-2} \right )^2. 
\end{align}

By \eqref{eqn_conv}, any hypersurface of $\mathbb{S}^{n+1}_K$ satisfying $\vert A\vert^2-\frac{1}{n-1}H^2\le 0$ is necessarily locally convex. Moreover, if $\vert A\vert^2-\frac{1}{n-1}H^2=0$, then $\lambda_1=0$, while $\lambda_2=\dots=\lambda_n=:\lambda$. When $K=0$ (i.e. $\mathbb{S}^{n+1}_K=\R^{n+1}$), this implies that the hypersurface is locally isometric to $\R\times \mathbb{S}^{n-1}_{\lambda^2}$.

\textcolor{red}{\texttt{What is the relationship between quadratic pinching and $k$-convexity?}}

\end{comment}

\subsection{Inscribed/exscribed curvature pinching}

We now present a noncollapsing estimate for mean curvature flow under the quadratic pinching condition \eqref{eq:strict quadratic pinching}. As mentioned in the introduction, this will not actually be required to obtain the main result (Theorem \ref{thm:main theorem}).

As in \cite{ALM13,AndrewsHanLiWei}, we  define the \emph{inscribed} and \emph{exscribed curvatures} $\overline k$ and $\underline k$ of an embedded hypersurface $\M\hookrightarrow\mathbb S_K^{n+1}$ by
\[
\overline k(p):=\sup_{q\neq p}\frac{2\left\langle p-q,\nu(p)\right\rangle}{\left\Vert p-q\right\Vert^2}\;\;\text{and}\;\; \underline k(p):=\inf_{q\neq p}\frac{2\left\langle p-q,\nu(p)\right\rangle}{\left\Vert p-q\right\Vert^2},
\]
respectively, where the inner product and norm are those of $\R^{n+2}$. Note that, under orientation reversal, $\overline k\mapsto-\underline k$ and $\underline k\mapsto-\overline k$. Observe also that
\[
\overline k(p)\ge\limsup_{q\to p}\frac{2\left\langle p-q,\nu(p)\right\rangle}{\left\Vert p-q\right\Vert^2}=\lambda_n
\]
and
\[
\underline k(p)\le\liminf_{q\to p}\frac{2\left\langle p-q,\nu(p)\right\rangle}{\left\Vert p-q\right\Vert^2}=\lambda_1\,.
\]
In particular,
\[
\overline k\ge \tfrac{1}{n}H\;\;\text{and}\;\;\underline k\le \tfrac{1}{n}H\,.
\]

In \cite{AndrewsHanLiWei} (cf. \cite[Proposition 2.1]{AL16}), it was shown that
\begin{equation}
(\pd_t-\Delta)\overline k\le (\vert A\vert^2+nK)\overline k-2nK\big(\overline k-\tfrac{1}{n}H\big)
\end{equation}
and
\begin{equation}
(\pd_t-\Delta)\underline k\ge (\vert A\vert^2+nK)\underline k+2nK\big(\tfrac{1}{n}H-\underline k\big)
\end{equation}
in the viscosity sense along a solution to mean curvature flow.

By the calculations in \S \ref{ssec:quadratic pinching}, the function
\begin{equation}\label{eq:noncollapsing F}
F:=\sqrt{4K+\frac{1}{n-2}H^2-\vert A\vert^2}
\end{equation}
is positive and satisfies
\begin{equation}\label{eq:positive supersolution}
(\pd_t-\Delta)F\ge (\vert A\vert^2+nK)F-(2n-4)KF
\end{equation}
on a solution to mean curvature flow which initially satisfies the pinching condition \eqref{eq:uniform quadratic pinching}. 
%
\begin{comment}
Similar calculations show that the function
\begin{equation}\label{eq:noncollapsing G}
G:=\sqrt{aH^2+bK}\,,
\end{equation}
where
\[
a:=\frac{1}{n-2+\alpha}-\frac{1}{n-1}+\frac{\alpha}{2n(n-1)}\;\;\text{and}\;\; b:=2(2-\alpha)\,,
\]
is positive and satisfies
\[
(\pd_t-\Delta)G\le (\vert A\vert^2+nK)G
\]
on a solution to mean curvature flow which initially satisfies the pinching condition \eqref{eq:uniform quadratic pinching}. 
\end{comment}
%
Thus, for such a solution,
\begin{align*}
(\pd_t-\Delta)\frac{\overline k}{F}\le{}&-4K\left(\frac{\overline k}{F}-\frac{1}{2}\frac{H}{F}\right)+2\left\langle\cd\frac{\overline k}{F},\cd\log F\right\rangle\\
\le{}&-4K\left(\frac{\overline k}{F}-C\right)+2\left\langle\cd\frac{\overline k}{F},\cd\log F\right\rangle
\end{align*}
\begin{comment}
and
\begin{align*}
(\pd_t-\Delta)\frac{\underline k}{G}\ge{}&2\left\langle\cd\frac{\underline k}{G},\cd\log G\right\rangle
\end{align*}
\end{comment}
in the viscosity sense, where
\[
C^2:=\frac{(n-2)(n-2+\alpha)}{4\alpha}.
\]
Since the pinching condition is invariant under orientation reversal, the maximum principle then yields the following noncollapsing estimates.

\begin{proposition}\label{prop:noncollapsing}
Let $X:\M\times[0,T)\to\mathbb S_K^{n+1}$, $n\ge 3$, be a solution to mean curvature flow such that $X_0:\M\to\mathbb S_K^{n+1}$ is embedded and satisfies \eqref{eq:uniform quadratic pinching} with $\alpha\in(0,1)$ when $n\ge 4$ or $\alpha\in(\frac{2}{3},1)$ if $n=3$. If
\[
\max_{\M\times\{0\}}\frac{\underline k}{F}\ge-\mu\;\;\text{and}\;\; \max_{\M\times\{0\}}\frac{\overline k}{F}\le \mu
\]
for some $\mu\ge C:=\sqrt{\frac{(n-2)(n-2+\alpha)}{4\alpha}}$, then
\begin{equation}\label{eq:noncollapsing}
\frac{\underline k}{F}(p,t)\ge -C-(\mu-C)\,\mathrm{e}^{-4Kt}\;\;\text{and}\;\;\frac{\overline k}{F}(p,t)\le C+(\mu-C)\,\mathrm{e}^{-4Kt}
\end{equation}
for all $(p,t)\in \M\times[0,T)$.
\end{proposition}

\begin{corollary}\label{cor:noncollapsing}
For any solution $X:\M\times[0,T)\to\mathbb S_K^{n+1}$ to mean curvature flow in $\mathbb S_K^{n+1}$, $n\ge 3$, with embedded initial condition satisfying \eqref{eq:uniform quadratic pinching}, there is a constant $\mu=\mu(n,\alpha,\displaystyle\min_{\M\times\{0\}}\tfrac{\underline k}{\vert H\vert+\sqrt{K}},\displaystyle\max_{\M\times\{0\}}\tfrac{\overline k}{\vert H\vert+\sqrt{K}})<\infty$ such that 
\[
\underline k\ge -\mu\big(\vert H\vert+\sqrt{K}\big)\;\;\text{and}\;\; \overline k\le \mu\big(\vert H\vert+\sqrt{K}\big)\,.
\]
\end{corollary}

\subsection{The surgery class}  

Given $n\ge 3$, $K>0$, $\alpha\in(0,1)$, $V<\infty$ and $\Theta<\infty$, we shall work with the class $\mathcal{C}^n_K(\alpha,V,\Theta)$ of hypersurfaces $X:\M\to\mathbb{S}^{n+1}_K$ satisfying
\begin{enumerate}
\item $\displaystyle \max_{\mathcal{M}^n\times\{0\}}\big(|A|^2-\tfrac{1}{n-2+\alpha}H^2\big)\le 2(2-\alpha)K^2$,
\item $\displaystyle \mu_0(\mathcal{M}^n)\le VK^{-\frac{n}{2}}$, and
\item $\displaystyle \max_{\mathcal{M}^n\times\{0\}}H^2\leq \Theta K$,
\end{enumerate}
where $\mu_t$ is the measure induced by $X(\cdot,t)$. Every properly immersed hypersurface of $\mathbb{S}^{n+1}_K$ which satisfies the strict quadratic pinching condition \eqref{eq:strict quadratic pinching} lies in the class $\mathcal{C}^n_K(\alpha,V,\Theta)$ for some choice of parameters $\alpha$, $V$, and $\Theta$ (with $\alpha>\frac{2}{3}$ when $n=3$). The first two conditions are preserved under mean curvature flow; the third is not. However, it is possible to preserve the class $\mathcal{C}^n_K(\alpha,V,\Theta)$ under mean curvature flow-with-surgery in the following sense. Given an initial hypersurface in the class $\mathcal{C}^n_K(\alpha,V,\Theta)$, we will be able to choose $\Theta_3>\Theta_2>\Theta_1$ (depending only on $n$, $\alpha$, $V$ and $\Theta$) such that, if we stop the flow when $H^2$ reaches the threshold $\Theta_3K$, then regions of squared mean curvature at least $\Theta_1K$ form `necks' of such quality that they may be replaced by high quality `caps', resulting in a new hypersurface in the class $\mathcal{C}^n_K(\alpha,V,\Theta_2)$. Since the estimates pass, with the same constants, to flows modified by such surgeries, the procedure can be repeated with the same constants each time $H^2$ reaches the threshold $\Theta_3K$.

If, in addition, the initial datum $X_0:\M\to\mathbb S_K^{n+1}$ is an embedding, then we can also find $\mu>0$ such that the inscribed and exscribed curvatures are $\mu$-pinched, in the sense that
\[
\min_{\M\times\{0\}}\frac{\underline k}{\vert H\vert+\sqrt{K}}\ge -\mu\;\;\text{and}\;\; \max_{\M\times\{0\}}\frac{\overline k}{\vert H\vert+\sqrt{K}}\le \mu\,.
\]
These inequalities are preserved under mean curvature flow in the sense of Corollary \ref{cor:noncollapsing}.

\begin{remark}
The class $\mathcal{C}^n_K(\alpha,V,\Theta)$ is defined slightly differently than the class $\mathcal{C}(R,\alpha_0,\alpha_1,\alpha_2)$ of hypersurfaces of $\R^{n+1}$ introduced in \cite{HuSi09}. We have found it convenient to use $K$ as the scale parameter here; formally, it is related to the scale parameter $R$ by the equation $R^{-2}=\Theta K$. The parameter $\alpha$ corresponds to $\alpha_0$ and $V$ corresponds to $\alpha_3\Theta^{\frac{n}{2}}$. The parameter $\alpha_1$ (a scale covariant lower bound for the mean curvature) is not required here --- it is needed in \cite{HuSi09} in order to bound the time of existence $T$ from above. Our estimates do not depend on a bound for $T$ and, indeed, some components of the flow may exist for all time.
\end{remark}

\section{The key estimates for smooth flows}

\subsection{The cylindrical estimate}\label{ssec:cylindrical estimate}

The following estimate provides a suitable analogue of the Huisken--Sinestrari ``cylindrical estimate'' \cite[Theorem 5.3]{HuSi09}.
 
\begin{theorem}[Cylindrical estimate (Cf. \cite{Hu84,Hu87,HuSi09,Ng15})]\label{thm:cylindrical estimate}
Let $X:\M\times[0,T)\to\mathbb{S}_K^{n+1}$, $n\geq 3$, be a solution to mean curvature flow with initial condition in the class $\mathcal{C}_K^n(\alpha,V,\Theta)$. Assume, further, that $\alpha>\frac{2}{3}$ when $n=3$. There exist $\delta=\delta(n,\alpha)>0$, $\eta_0=\eta_0(n,\alpha)>0$ and, for every $\eta\in(0,\eta_0)$, $C_\eta=C_\eta(n,\alpha,V,\Theta,\eta)<\infty$ such that
\begin{align}\label{eq:cylindrical estimate}
|A|^2-\frac{1}{n-1}H^2 \leq \eta H^2 + C_\eta K\vts\mathrm{e}^{-2\delta Kt} \quad\text{in}\quad \mc M^n\times[0,T)\,.
\end{align}
\end{theorem}
%Note that the constant $C_\eta$ does not depend on $K$. Thus, by a standard scaling argument,  it suffices to prove the estimate when $R=1$, which we henceforth assume. 

\begin{remark}
Note that, unlike the Euclidean analogue \cite[Theorem 5.3]{HuSi09}, the constant $C_\eta$ does not depend on a bound for the maximal time (which is controlled by the minimum of $H$ at the initial time in \cite{HuSi09}), and the zeroth order term becomes negligible for large times. This is because, as in \cite[2.1 Theorem]{Hu87}, the coercive term in \eqref{eq:ineqn_evol} gives rise to an exponential decay term in the ``$L^2$-estimate'', which we are able to exploit. This observation is useful here since, unlike in the Euclidean setting, solutions satisfying the initial conditions (1)-(3) can exist for all time (e.g. stationary hyperequators).
\end{remark}

\begin{remark}\label{rem:cylindrical apps}
As a Corollary of the cylindrical estimate, we find that any \emph{minimal} hypersurface of $\mathbb{S}_K^{n+1}$, $n\ge 3$, satisfying \eqref{eq:strict quadratic pinching} must be totally geodesic. This is a special case of Simons' Theorem \cite{Si68}. More generally, any hypersurface $X_0:\M\to\mathbb S_{K}^{n+1}$ satisfying \eqref{eq:strict quadratic pinching} which flows for all time under mean curvature flow (e.g. if $X_0(\M)$ is embedded and divides the area of $\mathbb S_K^{n+1}$ in two) must converge to a hyperequator.

We also find that $\vert A\vert^2-\frac{1}{n-1}H^2\le 0$ for any uniformly quadratically pinched ancient solution $X:\M\times(-\infty,0)\to \mathbb{S}^{n+1}_K$ with uniformly bounded area and curvature as $t\to-\infty$. A theorem of Huisken and Sinestrari \cite[Theorem 6.1]{HuSi15} then implies that $X:\M\times(-\infty,0)\to \mathbb{S}^{n+1}_K$ is either a stationary hyperequator or a shrinking hyperparallel.
\end{remark}

\begin{corollary}\label{cor:rigidity ancient}
Let $X:\M\times(-\infty,\omega)\to \mathbb S_K^{n+1}$ be an ancient solution to mean curvature flow. If
\begin{itemize}
\item $\displaystyle\limsup_{t\to-\infty}\max_{\M\times\{t\}}(\vert A\vert^2-\tfrac{1}{n-2}H^2)<4K$,
\item $\displaystyle\limsup_{t\to-\infty}\mu_t(\M)<\infty$, and
\item $\displaystyle\limsup_{t\to-\infty}\max_{\M\times\{t\}}H^2<\infty$,
\end{itemize}
then $X:\M\times(-\infty,\omega)\to \mathbb S_K^{n+1}$ is either a stationary hyperequator or a shrinking hyperparallel.
\end{corollary}

%By a scaling argument, it suffices to prove the estimate for fixed $R$ ($R=1$, say), which we henceforth assume. 

Motivated by \cite[2.1 Theorem]{Hu87} and \cite[Theorem 5.3]{HuSi09}, we will prove the estimate \eqref{eq:cylindrical estimate} by obtaining a bound for the function
\begin{align*}
f_{\sigma, \eta}:= \left[|A|^2-\left(\frac{1}{n-1}+\eta\right)H^2\right]W^{\sigma-1}
\end{align*}
for some $\sigma\in(0,1)$ and any $\eta\in \left(0,\frac{(2-\alpha)(n-\alpha)}{2n(n-1)(n-2+\alpha)}\right)$, where, setting
\[
a:=\frac{1}{n-2+\alpha}-\frac{1}{n-1}-\eta+\frac{\alpha}{2n(n-1)}\;\;\text{and}\;\; b:=2(2-\alpha)\,,
\]
the function $W$ is defined by
\[
W:=aH^2+bK\,.
\]
Observe that $W>0$ and $f_{\sigma,\eta}\leq W^\sigma$. The final term in the constant $a$ is chosen in order to obtain the good gradient and reaction terms in the following lemma. %This follows since 
%\begin{align*}
% |A| ^2 &- \left(\frac{1}{n-1} + \eta \right)  |H| ^2 - a H^2 -bK \\
% & =|A|^2 - \left( \alpha_n - \left( \frac{1}{ n-1} + \eta \right) + \frac{ \alpha}{ 2 n (n-1) } - \left(\frac{1}{n-1} + \eta \right)\right)H^2 - bK\\
% &=|A|^2 - \left ( \alpha_n+ \frac{ \alpha}{ 2 n (n-1) } \right)H^2- \beta_n K \\
% &\leq \alpha_n H^2 + \beta _n K - \left ( \alpha_n+ \frac{ \alpha}{ 2 n (n-1) } \right)H^2 - \beta_n K \\
% &= -  \frac{ \alpha}{ 2 n (n-1) } H^2  \leq  0.
%\end{align*}
%Hence $ f _0 \leq 1$.

%We will need an evolution equation for $f_{\sigma,\eta}$. 
\begin{lemma}\label{lem:evol_fsigma}
There exists $\delta=\delta(n,\alpha)>0$ such that
\begin{align}\label{eq:ineqn_evol}
(\partial_t-\Delta)f_{\sigma,\eta}\leq{}&2\sigma(|A|^2+nK)f_{\sigma,\eta}-4\delta Kf_{\sigma,\eta}\nonumber\\
{}&-2\delta f_{\sigma,\eta}\frac{|\nabla A|^2}{W}+2(1-\sigma)\left\langle \nabla f_{\sigma,\eta}, \frac{\nabla W}{W}\right\rangle
\end{align}
wherever $f_{\sigma,\eta}>0$.
%:=1-\tfrac{n+2}{3}\left(\tfrac{1}{n-2+\alpha}+\tfrac{\alpha}{2n(n-1)}\right)>0$ for $\alpha\leq \alpha(n)$ sufficiently small.
\end{lemma}
\begin{proof}
Set $f_\eta:=|A|^2-(\tfrac{1}{n-1}+\eta)H^2$. Basic manipulations (independent of the precise form of $f_\eta$ and $W$) yield
\begin{align*}
(\partial_t-\Delta)f_{\sigma,\eta}={}&W^{\sigma-1}(\partial_t-\Delta)f_\eta-(1-\sigma)f_{\sigma,\eta}W^{-1}(\partial_t-\Delta)W\\
{}&+2(1-\sigma)\left\langle\nabla f_{\sigma,\eta},\frac{\nabla W}{W}\right\rangle-\sigma(1-\sigma)f_{\sigma,\eta}\frac{|\nabla W|^2}{W^2}\,.
\end{align*}
The final term will be discarded.

Applying \eqref{evol_Asquare} and \eqref{evol_Hsquare}, we compute
\begin{align*}
(\partial_t-\Delta)W=2\left(|A|^2+nK\right)(W-bK)-2a|\nabla H|^2
\end{align*}
and
\begin{align}\label{eq:evolve_f_eta}
(\partial_t-\Delta)f_\eta={}&2\left(|A|^2+nK\right)f_\eta-4nK\left(|A|^2-\tfrac{1}{n}H^2\right)\nonumber\\
{}&-2\left(|\nabla A|^2-\left(\tfrac{1}{n-1}+\eta\right)|\nabla H|^2\right)\,.%\\
%\leq{}&2\left(|A|^2+nK\right)f_\eta-2\gamma|\nabla A|^2\,.
\end{align}

Combining the preceding three identities and estimating $f_{\sigma,\eta}\leq W^{\sigma}$, $|\nabla H|^2\leq \frac{n+2}{3}|\nabla A|^2$ and $\sigma<1$ yields
\begin{align*}
(\partial_t-\Delta)f_{\sigma,\eta}\leq{}&2\sigma(|A|^2+nK)f_{\sigma,\eta}+2(1-\sigma)\left\langle\nabla f_{\sigma,\eta},\frac{\nabla W}{W}\right\rangle\\
{}&+2W^{\sigma-1}\left(bK(|A|^2+nK)\frac{f_\eta}{W}-2nK(|A|^2-\tfrac{1}{n}H^2)\right)\\
{}&-2f_{\sigma,\eta}\left(1-\frac{n+2}{3}\left[\frac{1}{n-1}+\eta+a\right]\right)\frac{|\nabla A|^2}{W}\,.
%-4K\frac{f_{\sigma,\eta}}{W}\left((n-2+\alpha)|A|^2-H^2-n(2-\alpha)K)\right)\,.
\end{align*}

\begin{comment}
Since $n\geq 4$, we can find $\alpha_0(n)>0$ such that
\[
\frac{1}{n-2+\alpha}+\frac{\alpha}{2n(n-1)}+\frac{3\alpha}{n+2}\leq\frac{3}{n+2}
\]
for $\alpha\leq \alpha_0(n)$ and hence
\[
\frac{1}{n-1}+\eta+a=\frac{1}{n-2+\alpha}+\frac{\alpha}{2n(n-1)}\leq\frac{3}{n+2}(1-\alpha)\,.
\]
\end{comment}

Consider the term
\begin{align*}
Z:={}&(2-\alpha)\left(|A|^2+nK\right)\frac{f_\eta}{W}+H^2-n|A|^2\\
\leq{}&(2-\alpha)\left(\frac{1}{n-2+\alpha}H^2+bK+nK\right)\frac{f_\eta}{W}+H^2-n|A|^2\,.
\end{align*}
Noting that
\[
na=(2-\alpha)\left(\frac{1}{n-2+\alpha}-\frac{1}{2(n-1)}\right)-n\eta\,,
\]
and estimating $f_\eta\le W$, we find
\begin{align*}
Z\leq{}&%\left(naH^2+\left[\frac{2-\alpha}{2(n-1)}+n\eta\right]\!H^2+\left[2-\alpha+\frac{n}{2}\right]\!bK\!\right)\!\frac{f_\eta}{W}+H^2-n|A|^2\\
%={}&
nf_\eta+\left(\left[\frac{2-\alpha}{2(n-1)}+n\eta\right]H^2+\left[2-\alpha-\frac{n}{2}\right]bK\right)\frac{f_\eta}{W}+H^2-n|A|^2\\
\le{}&-\left(\frac{1}{2(n-1)}H^2+\left[\alpha+\frac{n}{2}-2\right]bK\right)\frac{f_\eta}{W}\,.
\end{align*}
%Since $n\geq 4$, we may discard the term $(2-n/2)b f_\eta/W$. Estimating $f_\eta\leq W$ then yields
%\begin{align*}
%Z\leq{}&%\left(1-n\eta-\frac{n}{n-1}\right)H^2+\left(\frac{2}{2(n-1)}+n\eta\right)H^2+\left(2-\frac{n}{2}\right)b
%-\alpha\left(\frac{1}{2(n-1)}H^2+bK\right)\frac{f_\eta}{W}\leq -\alpha f_\eta\,.
%\end{align*}
Now set
\bann
\delta:={}&\min\left\{1-\frac{n+2}{3}\left[\frac{1}{n-2+\alpha}-\frac{\alpha}{2n(n-1)}\right],\frac{a^{-1}}{2(n-1)},\alpha+\frac{n}{2}-2\right\}\\
>{}&0\,.\qedhere
\eann

%observe that, when $n\geq 4$ and $f_{\sigma,\eta}\geq 0$,
%\begin{align*}
%|A|^2-\frac{1}{n-2+\alpha}H^2-\frac{n(2-\alpha)}{n-2+\alpha}K={}& |A|^2-\left(\frac{1}{n-1}+\eta\right)H^2-aH^2-\frac{n(2-\alpha)}{n-2+\alpha}K\\
%\geq{}& -W+(2-\alpha)K\left(2-\frac{n}{n-2+\alpha}\right)\geq-W\,.
%\end{align*}
\end{proof}
%Thus, if $\sigma\leq\frac{\alpha\gamma}{2n}$, then
%\begin{align*}
%(\partial_t-\Delta)f_{\sigma,\eta}\leq{}&2\sigma|A|^2f_{\sigma,\eta}-\alpha \gamma Kf_{\sigma,\eta}-2\gamma f_{\sigma,\eta}\frac{|\nabla A|^2}{W}+2(1-\sigma)\left\langle \nabla f_{\sigma,\eta}, \frac{\nabla W}{W}\right\rangle\,.
%\end{align*}

We wish to bound $\mathrm{e}^{2\delta Kt}f_{\sigma,\eta}$ from above. It will suffice to consider points where $f_{\sigma, \eta}>0$. To that end, we consider the function
\[
f_+%:=(\mathrm{e}^{\gamma Kt}f _{\sigma, \eta})_{+}
:=\max\{\mathrm{e}^{2\delta Kt}f_{\sigma,\eta},0 \}.
\]

\begin{lemma}
There exist constants $\ell=\ell(n,\alpha,\eta)<\infty$ and $C=C(n,K,\alpha,V,\Theta,\sigma,p)$ such that
\begin{align}\label{eq:Lpest}
\int f_+^p\,d\mu\leq C\vts\mathrm{e}^{-\delta pKt}
\end{align}
so long as $\sigma\leq \ell p^{-\frac{1}{2}}$ and $p>\ell^{-1}$.
\end{lemma}
\begin{proof}
Applying \eqref{eq:ineqn_evol} yields
\begin{align}\label{eq:basicesimate}
\frac{d}{dt}\int f_+^p\,d\mu={}&p\int f_+^{p-1}\partial_tf_{\sigma,\eta}\,d\mu-\int f_+^pH^2\,d\mu\nonumber\\
\leq{}&-p(p-1)\int f_+^{p-2}|\nabla f_{\sigma,\eta}|^2\,d\mu-2\delta p\int f_+^p\frac{|\nabla A|^2}{W}\,d\mu\nonumber\\
{}&+2\sigma p\int f_+^p|A|^2\,d\mu-2\delta pK\int f_+^p\,d\mu\nonumber\\
{}&+2p\int f_+^p\frac{|\nabla f_{\sigma,\eta}|}{f_{\sigma,\eta}}\frac{|\nabla W|}{W}\,d\mu\nonumber\\
\leq{}&-p(p-p^{\frac{1}{2}}-1)\int f_+^{p-2}|\nabla f_{\sigma,\eta}|^2\,d\mu\nonumber\\
{}&-(2\delta p-4anp^{\frac{1}{2}})\int f_+^p\frac{|\nabla A|^2}{W}\,d\mu\nonumber\\
{}&+2\sigma p\int f_+^p|A|^2\,d\mu-2\delta pK\int f_+^p\,d\mu\,.
\end{align}
%where $c:=\sigma n+2(n-2+\alpha)$ and we applied Young's inequality in the final line.

To estimate the penultimate term, we apply Proposition \ref{prop:Poincare}. Setting $u^2=f_+^p$ and $r=p^{\frac{1}{2}}$, this yields
\begin{align*}
2\int f_+^p|A|^2\,d\mu\leq{}& C\int f_+^pW\,d\mu\\
\leq{}& C\int f_+^p\left(p^{\frac{3}{2}}\frac{|\nabla f_{\sigma,\eta}|}{f_{\sigma,\eta}^2}+p^{\frac{1}{2}}\frac{|\nabla A|^2}{W}+K\right)\,d\mu\,,
\end{align*}
where $C$ depends on $n$, $\alpha$ and $\eta$, and we thereby arrive at
\begin{align*}
\frac{d}{dt}\int f_+^p\,d\mu\leq{}&-p(p-C\sigma p^{\frac{3}{2}}-p^{\frac{1}{2}}-1)\int f_+^{p-2}|\nabla f_{\sigma,\eta}|^2\,d\mu\\
{}&-(2\delta p-C\sigma p^{\frac{3}{2}}-4anp^{\frac{1}{2}})\int f_+^p\frac{|\nabla A|^2}{W}\,d\mu\\
{}&-(2\delta-C\sigma)pK\int f_+^p\,d\mu\,.
\end{align*}
Choosing $\sigma\leq \ell p^{-\frac{1}{2}}$ and $p\ge \ell^{-1}$, for $\ell=\ell(n,\alpha,\eta)$ sufficiently small, we can arrange that
\begin{align*}
\frac{d}{dt}\log\left(\int f_+^p\,d\mu\right)\leq{}&-\delta pK
\end{align*}
and hence
\begin{align}\label{eq:Lpest}
\int f_+^p\,d\mu\leq \mathrm{e}^{-\delta pKt}\int f_+^p(\cdot,0)d\mu_0\leq C\mathrm{e}^{-\delta pKt}\,,
\end{align}
where $C=C(n,K,\alpha,V,\Theta,\sigma,p)$.
\end{proof}

%Let us state this as a Lemma
%\begin{lemma}
%Let $X:\mc M^n\times [0,T)\to \mathbb{S}^{n+1}_K$ be a solution of mean curvature flow satisfying ... For any $\eta>0$ there is a constant $C=C(n,\alpha,|\mc M^n_0|,T,\eta,\sigma,p)$ such that
%\begin{align*}
%\int f_+^p\leq C(n,\alpha,\eta, |\mathcal{M}_0^n|,T,\sigma,p)
%\end{align*}
%whenever...
%\end{lemma}

This $L^2$-estimate (for $v:=f_+^{\frac{p}{2}}$) can be bootstraped to an $L^\infty$-estimate using Stampacchia iteration.

\begin{proof}[Proof of Theorem \ref{thm:cylindrical estimate}]
Given $k\geq 0$, consider
\begin{align*}
v_k^2:=\left(\mathrm{e}^{2\delta Kt}f_{\sigma,\eta}-k\right)^p_+\quad\text{and}\quad V_{k}(t):=\{x\in \mc M^n:v_k(x,t)>0\}
\end{align*}
and set
\begin{align*}
u(k):=\int_0^T\!\!\!\int v_k^2\,d\mu\,dt\quad\text{and}\quad U(k):=\int_0^T\!\!\!\int_{V_{k}}d\mu\,dt\,.
\end{align*}
Note that, for any $h>k>0$,
%\bann
%a(h,\rho)%=\int_{-T}^0\hspace{-1mm}\int_{A_{h,\rho}}\frac{(\Ges-k)_+^p}{(\Ges-k)_+^p}
%\leq\int_{-T}^0\hspace{-1mm}\int_{A_{h,\rho}}\frac{(\Ges-k)_+^p}{(h-k)^p}\,d\mu\,dt
%\eann
%so that
\begin{align}\label{eq:iteration1}
(h-k)^pU(h)\leq u(k)\,.
\end{align}

So we need an estimate for $u(k)$. Computing as in \eqref{eq:basicesimate}, we can estimate
\begin{align}\label{eq:evolvevk}
\frac{d}{dt}\int v_k^2\,d\mu+\int_{V_{k}}|\nabla v_k|^2\,d\mu+\int_{V_{k}}H^2v_k^2\,d\mu\leq{}&2\sigma p\int_{V_{k}}f_+^p(|A|^2+nK)\,d\mu\nonumber\\
\leq{}&c_n\sigma p\int_{V_{k}}f_+^pW\,d\mu
\end{align}
for any $k>0$ if $p>2$, where $c_n$ depends only on $n$.

We shall exploit the good gradient term using the Sobolev inequality (Theorem \ref{thm:Sobolev}). Indeed, since \eqref{eq:Lpest} implies that
\[
|V_k|\leq k^{-p}\int f_{+}^p\,d\mu\leq Ck^{-p}\,,
\]
we can apply \eqref{eq:Sobolev} to obtain
\[
\frac{1}{c_n}\left(\int v_k^{2^\ast}d\mu\right)^{\frac{1}{2^\ast}}\leq \int\big(|\nabla v_k|^2+H^2v_k^2\big)\,d\mu%+\left(\int_{V_k} |H|^nd\mu\right)^{\frac{2}{n}}\left(\int v_k^{2^\ast}d\mu\right)^{\frac{2}{2^\ast}}\,.
\]
so long as
\[
k>k_0:=\left(\frac{(n+1)C}{\omega_n}K^{\frac{n}{2}}\right)^p.
\]
\begin{comment}
Setting $\sigma':=\sigma+\frac{n}{2p}$, we can estimate, using \eqref{eq:Lpest} (choosing $\ell$ slightly smaller if necessary), 
\begin{align*}
\int_{V_{k}}|H|^n\,d\mu\leq{}& k^{-p}\int_{V_{k}}|H|^nf_+^p\,d\mu\\
\leq{}& k^{-p}a^{-\frac{n}{2}}\int_{V_{k}}W^{\frac{n}{2}}f_+^p\,d\mu\\
={}&k^{-p}a^{-\frac{n}{2}}\int_{V_k}(f_{\sigma',\eta})_+^p\,d\mu \\
\leq{}&Ck^{-p},
\end{align*}
where $C=C(n,K,V,\Theta,\sigma,p)$. Thus, choosing $k_0$ larger if necessary, we can arrange that
\begin{equation}\label{eq:sobolev}
\left(\int v_k^{2^\ast}d\mu\right)^{\frac{1}{2^\ast}}\leq c_n\int |\nabla v_k|^2d\mu
\end{equation}
when $k>k_0$.
\end{comment}
%\[
%f^p_{\sigma',\eta}=f_{\eta}^pW^{\sigma p-p+\frac{n}{2}}=f^p_{\sigma,\eta} W^\frac{n}{2}\geq a^{\frac{n}{2}}f^p_{\sigma,\eta}H^n
%\]
%
%
Recalling \eqref{eq:evolvevk}, we arrive at
\begin{equation}\label{eq:v_k monotonicity}
\frac{d}{dt}\int v_k^2\,d\mu+\left(\int v_k^{2^\ast}d\mu\right)^{\frac{1}{2^\ast}}\leq c_n\sigma p\int_{V_{k}}f_+^pW\,d\mu
\end{equation}
for $k>k_0$. Assuming that $k_0>\sup_{\mathcal M\times\{0\}}f_{\sigma,\eta}$, integration then yields
\begin{equation}\label{eq:still holds with surgeries}
\sup_{[0,T)}\int v_k^2\,d\mu+\int_0^T\hspace{-2mm}\left(\int v_k^{2^\ast}d\mu\right)^{\frac{1}{2^\ast}}\leq c_n\sigma p\int_0^T\hspace{-2mm}\int_{V_{k}}f_+^pW\,d\mu
\end{equation}
for $k>k_0$.

Using the interpolation inequality, we can estimate
\begin{align*}
\int v_k^{\frac{2(n+2)}{n}}d\mu\leq \left(\int v_k^2\,d\mu\right)^{\frac{2}{n}}\left(\int v_k^{2^\ast}\,d\mu\right)^{\frac{2}{2^\ast}}
\end{align*}
so that, by Young's inequality,
\begin{align}\label{eq:uest1}
\left(\int_{0}^T\hspace{-2mm}\int v_k^{\frac{2(n+2)}{n}}d\mu\,dt\right)^{\frac{n}{n+2}}\!\!\leq{}& \left(\sup_{t\in [0,T)}\int v_k^2\,d\mu\right)^{\frac{2}{n+2}}\!\!\left(\int_{0}^T\!\!\left(\int v_k^{2^\ast}d\mu\right)^{\frac{2}{2^\ast}}\!\!dt\right)^{\frac{n}{n+2}}\nonumber\\
\leq{}&\tfrac{2}{n+2}\sup_{t\in [0,T)}\int v_k^2\,d\mu+\tfrac{n}{n+2}\int_{0}^T\!\!\left(\int v_k^{2^\ast}d\mu\right)^{\frac{2}{2^\ast}}\!dt\nonumber\\
\leq{}& c_n\sigma p\int_0^T\hspace{-2mm}\int_{V_{k}}f_+^pW\,d\mu\,.
\end{align}

Set $\sigma':=\sigma+\frac{1}{p}\lesssim p^{-\frac{1}{2}}$. Given $r>1$ (to be determined momentarily), we may choose $\ell$ slightly smaller if necessary, depending now also on $r$, so that H\"older's inequality and the $L^2$-estimate \eqref{eq:Lpest} yield
\begin{align}\label{eq:uest2}
\int_0^T\hspace{-2mm}\int_{V_k}f_+^pW\,d\mu\,dt\leq{}&U(k)^{1-\frac{1}{r}}\left(\int_{0}^T\hspace{-2mm}\int_{V_k}f_+^{pr}W^{r}\,d\mu\,dt\right)^{\frac{1}{r}}\nonumber\\
={}& U(k)^{1-\frac{1}{r}}\left(\int_0^T\hspace{-2mm}\int_{V_k}(f_{\sigma',\eta})_{+}^{pr}\,d\mu\,dt\right)^{\frac{1}{r}}\nonumber\\
\le{}& U(k)^{1-\frac{1}{r}}\left(C\int_0^T\mathrm{e}^{-\delta prKt}\,dt\right)^{\frac{1}{r}}\nonumber\\
={}& \left[\frac{C}{\delta prK}\left(1-\mathrm{e}^{-\delta pr KT}\right)\right]^{\frac{1}{r}}\,U(k)^{1-\frac{1}{r}}
\end{align}
for $p\geq \ell^{-1}$ and $\sigma\leq \ell p^{-\frac{1}{2}}$, where $C=C(n,K,\alpha,V,\Theta,\sigma,p)$.

Since, by H\"older's inequality,
\begin{align*}%\label{eq:uest3}
u(k)\leq{}& U(k)^{\frac{2}{n+2}}\left(\int_0^T\hspace{-2mm}\int v_k^{\frac{2(n+2)}{n}}d\mu\,dt\right)^{\frac{n}{n+2}},
\end{align*}
the estimates \eqref{eq:uest1} and \eqref{eq:uest2} yield
\begin{align}\label{eq:iteration2}
u(k)\leq{}&C\,U(k)^{\frac{2}{n+2}+1-\frac{1}{r}}
\end{align}
for $k>k_0$, where $C=C(n,K,\alpha,V,\Theta,\sigma,p,r)$. We conclude from \eqref{eq:iteration1} that
\begin{align*}
(h-k)^p\,U(h)\leq{}&C\,U(k)^{\gamma}\,,
\end{align*}
where $\gamma:=1+\frac{2}{n+2}-\frac{1}{r}$ and $C=C(n,K,\alpha,V,\Theta,\sigma,p,r)$.

At this point, we fix $r>1+\frac{2}{n}$ (so that $\gamma>1$), $p^{-1}=\ell(n,\alpha,\eta)$, and $\sigma=\ell p^{-\frac{1}{2}}=\ell^{\frac{3}{2}}$. Stampacchia's Lemma \cite[Lemma 4.1]{St66} then yields
\begin{align*}
U(k_0+d)=0\,,
\end{align*}
where
\begin{align*}
d^p:=2^{p\gamma/(\gamma-1)}CU(k_0)^{\gamma-1}\,.
\end{align*}
Estimating via \eqref{eq:Lpest} (assuming $k_0\geq 1$)
\begin{align*}
U(k_0)\leq k_0^{-p}\int_0^T\hspace{-2mm}\int f_+^p\,d\mu\,dt\leq C(n,K,\alpha,V,\Theta,\eta)\,,
\end{align*}
we conclude that
\begin{align*}
\mathrm{e}^{2\delta Kt}f_{\sigma,\eta}\leq C(n,K,\alpha,V,\Theta,\eta)\,.
\end{align*}
Young's inequality then yields
\begin{align*}
|A|^2-\frac{1}{n-1}H^2\leq 2\eta H+C(n,K,\alpha,V,\Theta,\eta)\vts\mathrm{e}^{-2\delta K t}\,.
\end{align*}
The theorem follows by the scaling covariance of the estimate.
\end{proof}

\subsection{The gradient estimate}

Next, we derive a suitable analogue of the ``gradient estimate" \cite[Theorem 6.1]{HuSi09}. We need the following a priori interior estimates for solutions with initial data in the class $\mathcal{C}^n_K(\alpha,V,\Theta)$.

\begin{proposition}\label{prop:class C universal interior estimates}
Let $X:\M\times [0,T)\to\mathbb{S}_K^{n+1}$ be a maximal solution to mean curvature flow with initial condition in the class $\mathcal{C}_K^n(\alpha,V,\Theta)$. Defining $\Lambda_0$ and $\lambda_0$ by
\begin{equation}\label{eq:lambdas}
\Lambda_0/2:=\frac{1}{n-2+\alpha}\Theta^2+2(2-\alpha)\;\;\text{and}\;\;\mathrm{e}^{2n\lambda_0}:=1+\frac{n}{n+\Lambda_0},
\end{equation}
we have
\begin{equation}\label{eq:class C universal T bound}
\mathrm{e}^{2nKT}\ge 1+\frac{2n}{\Lambda_0}\,,
\end{equation}
and
\begin{equation}\label{eq:class C universal interior estimates}
\max_{\M\times\{\lambda_0K^{-1}\}}\vert\cd^k{A}\vert^2\leq \Lambda_kK^{k+1}
\end{equation}
for every $k\in\mathbb{N}$, where $\Lambda_k$ depends only on $n$, $k$ and $\Lambda_0$.
\end{proposition}
\begin{proof}
Since
\[
\max_{\M\times\{0\}}\vert{A}\vert^{2}\leq \Lambda_0K/2\,,
\]
a straightforward \textsc{ode} comparison argument applied to the inequality
\[
(\pd_t-\Delta)\vert{A}\vert^2\leq 2(\vert A\vert^2+nK)\vert{A}\vert^2
\]
yields
\[
\max_{\M\times\{t\}}\vert{A}\vert^{2}\leq\frac{nK}{\left(1+\frac{2n}{\Lambda_0}\right)\mathrm{e}^{-2nKt}-1}\,.
\]
We immediately obtain \eqref{eq:class C universal T bound} and
\ba\label{eq:curvature_bounded_gradient_estimate}
\vert{A}\vert^2(\vts\cdot\vts,t)\leq \Lambda_0K\;\;\text{for all}\;\; t\leq \lambda_0K^{-1}\,.
\ea
The claim \eqref{eq:class C universal interior estimates} now follows from the Bernstein estimates (Proposition \ref{prop:Bernstein}).
\end{proof}

Modifying an argument of Huisken \cite[Theorem 6.1]{Hu84} and Huisken--Sinestrari \cite[Theorem 6.1]{HuSi09}, we can now obtain a pointwise estimate for the gradient of the second fundamental form which holds up to the singular time.

\begin{theorem}[Gradient estimate (cf. {\cite[Theorem 6.1]{HuSi09}})]\label{thm:gradient estimate}
Let $X:\M\times[0,T)\to\mathbb{S}_K^{n+1}$, $n\geq 2$, be a solution to mean curvature flow with initial condition in the class $\mathcal{C}_K^n(\alpha,V,\Theta)$. There exist constants $\delta=\delta(n,\alpha)$, $c=c(n,\alpha,\Theta)<\infty$, $\eta_0=\eta_0(n)>0$ and, for every $\eta\in(0,\eta_0)$, $C_\eta=C_\eta(n,\alpha,V,\Theta,\eta)<\infty$ such that
\begin{equation}\label{eq:gradient estimate}
\vert\cd  A\vert^2\leq c\left[(\eta+\tfrac{1}{n-1})H^2-\vert{ A}\vert^2\right]W+C_\eta K^2\mathrm{e}^{-2\delta Kt}
\end{equation}
in $\M\times [\lambda_0K^{-1},T)$, where $\lambda_0$ is defined by \eqref{eq:lambdas}.
\end{theorem}

\begin{remark}
As for the cylindrical estimate, the constants do not depend on a bound for the maximal time, and the zeroth order term becomes negligible for large times.
\end{remark}

Note that the conclusion is not vacuous since, by Proposition \ref{prop:class C universal interior estimates}, the maximal existence time of a solution with initial data in the class $\mathcal{C}_K^n(\alpha,V,\Theta)$ is at least $\frac{1}{2nK}\log\left(1+\frac{2n}{\Lambda_0}\right)>\lambda_0 K^{-1}$. 

Setting $\eta=1$, say, yields the cruder estimate
\begin{equation}\label{eq:gradient estimate eta=1}
\vert\cd  A\vert^2\leq C(H^4+K^2),
\end{equation}
where $C=C(n,\alpha,V,\Theta)$.

\begin{proof}[Proof of Theorem \ref{thm:gradient estimate}]
We proceed as in \cite[Theorem 6.1]{HuSi09}. By \eqref{eqn_evolderiv},
\bann
(\partial_t- \Delta)|\nabla A|^2 \leq{}& -2 |\nabla^2 A|^2 + c_n\lb |A|^2+nK\rb|\nabla A|^2\,.
\eann
We will control the bad term using the good term in the evolution equation for $\vert A \vert^2$ and the Kato inequality \eqref{ineqn_Kato}.
 
By the cylindrical estimate, given any $\eta>0$ we can find a constant $C_\eta=C_\eta(n,\alpha,V,\Theta,\eta)>2$ such that
\[
\vert A\vert^2-\frac{1}{n-1}H^2\leq \eta H^2+C_\eta K\mathrm{e}^{-2\delta Kt}\,,
\]
and hence
\[
G_\eta:=2C_\eta K\mathrm{e}^{-2\delta Kt}+\left(\eta+\frac{1}{n-1}\right)H^2-\vert A\vert^2\geq C_\eta K\mathrm{e}^{-2\delta Kt}>0\,.
\]
Similarly, there is a constant $C_0=C_0(n,\alpha,V,\Theta)>2$ such that
\[
\vert A\vert^2-\frac{1}{n-1}H^2\leq \frac{2(n-2)}{(n+2)(n-1)}H^2+C_0K\,,
\]
which ensures that
\[
G_0:=2C_0K+\frac{3}{n+2} H^2-\vert A\vert^2\geq C_0K>0\,.
\]

By \eqref{eq:evolve_f_eta},
\bann
(\pd_t-\Delta)G_\eta={}&2(\vert A\vert^2+nK)(G_\eta-2C_\eta K\mathrm{e}^{-2\delta Kt})+ 4nK\left(\vert A\vert^2-\tfrac{1}{n}H^2\right)\\
{}&+2\left[\vert\cd  A\vert^2-\lb\eta+\tfrac{1}{n-1}\rb\vert\cd H\vert^2\right]-2\delta C_\eta K^2\mathrm{e}^{-2\delta Kt}\,.
\eann
Since $G_\eta\geq C_\eta K\mathrm{e}^{-2\delta Kt}$, we can estimate $G_\eta-2C_\eta K\mathrm{e}^{-2\delta Kt}\geq -G_\eta$. By the Kato inequality \eqref{ineqn_Kato}, we can estimate
\bann
\vert\cd  A\vert^2-\lb\tfrac{1}{n-1}+\eta\rb\vert\cd H\vert^2\ge{}& \tfrac{n+2}{3}\left[\tfrac{3}{n+2}-\tfrac{1}{n-1}-\eta\right]\vert\cd A\vert^2\\
\geq{}& \tfrac{\beta}{2}\vert\cd  A\vert^2\,,
\eann
where
\begin{equation}\label{eq:beta in grad est}
\beta := \frac{1}{2} \left( \frac{ 3}{n+2} - \frac{1}{n-1} \right),%=\frac{2n-5}{2(n+2)(n-1)},
\end{equation}
so long as $\eta\leq \left(1-\frac{3}{4(n+2)}\right)\beta$. Estimating, finally, 
\[
2\delta C_\eta K^2\mathrm{e}^{-2\delta Kt}\le 2\delta KG_\eta\,,
\]
we arrive at
\bann
(\pd_t-\Delta)G_\eta\geq{}&-2(\vert A\vert^2+nK)G_\eta+\beta\vert\cd  A\vert^2-2\delta KG_\eta\,.
\eann

Similarly,
\bann
(\pd_t-\Delta)G_0\geq{}&-2(\vert A\vert^2+nK)G_0\,.
\eann

We seek a bound for the ratio $\frac{\vert\cd  A\vert^2}{G_\eta G_0}$. Note that, at a local spatial maximum of $\frac{\vert\cd  A\vert^2}{G_\eta G_0}$,
\bann
0=\cd_k\frac{\vert\cd  A\vert^2}{G_\eta G_0}=2\frac{\inner{\cd_k\cd  A}{\cd  A}}{G_\eta G_0}-\frac{\vert\cd  A\vert^2}{G_\eta G_0}\lb\frac{\cd_k G_\eta}{G_\eta}+\frac{\cd_kG_0}{G_0}\rb.
\eann
In particular,
\bann
4\frac{\vert\cd  A\vert^2}{G_\eta G_0}\inner{\frac{\cd G_\eta}{G_\eta}}{\frac{\cd G_0}{G_0}}\leq{}&\frac{\vert\cd  A\vert^2}{G_\eta G_0}\left\vert\frac{\cd G_\eta}{G_\eta}+\frac{\cd G_0}{G_0}\right\vert^2\leq4\frac{\vert\cd^2  A\vert^2}{G_\eta G_0}\,.
\eann
%Let $(x_0,t_0)$ be a local parabolic maximum of $\frac{\vert\cd  A\vert^2}{G_\eta G_0}$. That is, there is some $r>0$ such that
%\[
%\frac{\vert\cd  A\vert^2}{G_\eta G_0}(x_0,t_0)=\max_{P_r(x_0,t_0)}\frac{\vert\cd  A\vert^2}{G_\eta G_0}\,,
%\]
%where, denoting by $B_r(x_0,t_0)$ the ball in $M$ centered at $x_0$ of radius $r$ with respect to the metric at time $t_0$, $P_r(x_0,t_0):= B_r(x_0,t_0)\times (t_0-r^2,t_0]$. 
Suppose that $\frac{\vert \cd A\vert^2}{G_\eta G_0}$ attains a parabolic interior local maximum at $(x_0,t_0)$. Then, at $(x_0,t_0)$,
\bann
0\leq{}& (\pd_t-\Delta)\frac{\vert\cd  A\vert^2}{G_\eta G_0}\\
={}&\frac{(\pd_t-\Delta)\vert\cd  A\vert^2}{G_\eta G_0}-\frac{\vert\cd  A\vert^2}{G_\eta G_0}\lb\frac{(\pd_t-\Delta)G_\eta}{G_\eta}+\frac{(\pd_t-\Delta)G_0}{G_0}\rb\\
{}&+\frac{2}{G_\eta G_0}\inner{\cd\frac{\vert\cd  A\vert^2}{G_\eta G_0}}{\cd(G_\eta G_0)}+2\frac{\vert\cd  A\vert^2}{G_\eta G_0}\inner{\frac{\cd G_\eta}{G_\eta}}{\frac{\cd G_0}{G_0}}\\
\leq{}&\frac{\vert\cd  A\vert^2}{G_\eta G_0}\lb (c_n+4)(\vert A\vert^2+nK)+2\delta K-\beta\frac{\vert\cd  A\vert^2}{G_\eta}\rb
\eann
and hence
\bann
\frac{\vert\cd  A\vert^2}{G_\eta G_0}\leq{}& \frac{(c_n+4)(\vert A\vert^2+nK)+2\delta K}{2C_0K+\frac{3}{n+2}H^2-\vert A\vert^2}\,.
\eann
Since
\[
\vert A\vert^2\le \frac{1}{n-2+\alpha}H^2+2(2-\alpha)K\,,
\]
we obtain, at $(x_0,t_0)$,
\bann
\frac{\vert\cd  A\vert^2}{G_\eta G_0}%\leq{}& \frac{C_n(H^2+K)}{2(C_0-(2-\alpha))K+\left(\frac{3}{n+2}-\frac{1}{n-2+\alpha}\right)H^2}\\
\le{}&C\,,
\eann
where $C$ depends only on $n$ and $\alpha$.

On the other hand, since $G_0>C_0K$ and $G_\eta>C_\eta K\mathrm{e}^{-2\delta Kt}$, if no interior local parabolic maxima are attained, then, by Proposition \ref{prop:class C universal interior estimates}, we have for any $t\ge \lambda_0K^{-1}$
\bann
\max_{\M\times \{t\}}\frac{\vert\cd{A}\vert^2}{G_0G_\eta}\leq{}&\max_{\M\times \{\lambda_0K^{-1}\}}\frac{\vert\cd{A}\vert^2}{G_0G_\eta}\\
\leq{}&\max_{\M\times \{\lambda_0K^{-1}\}}\frac{\vert\cd{A}\vert^2}{C_0C_\eta K^2\mathrm{e}^{-2\delta \lambda_0}}\\
\leq{}&\frac{\Lambda_1\mathrm{e}^{\lambda_0}}{C_0C_\eta}\\
\leq{}&\Lambda_1\mathrm{e}^{\lambda_0}\,.
\eann
The Theorem now follows from Young's inequality.
%\bann
%\vert\cd{A}\vert^2\le{}&c(2C_0K+\tfrac{3}{n+2}H^2)(2C_\eta K\mathrm{e}^{-2\alpha Kt}+(\tfrac{1}{n-1}+\eta)H^2-\vert A\vert^2)\\
%\le{}&C_\eta'K^2\mathrm{e}^{-2\alpha Kt}+C_0' K\mathrm{e}^{-2\alpha Kt}H^2\\
%{}&+C_0'(H^2+K)\big((\tfrac{1}{n-1}+\eta)H^2-\vert A\vert^2\big)\\
%\eann
\end{proof}

The following simple lemma illustrates the utility of scale-invariant, pointwise gradient estimates for the curvature.
\begin{lemma}\label{lem:consequence of gradient estimate}
Let $X:\M\to\mathbb{S}^{n+1}$ be an immersed hypersurface. If
\[
\sup_{\M}\frac{\vert\cd H\vert}{H^2}\leq c_\sharp<\infty\,,
\]
then
\begin{equation}
\frac{H(p)}{2}\leq H(q)\leq 2H(p)
\end{equation}
for all $q\in \mathcal{B}_{\frac{1}{2c_\sharp H(p)}}(p)$, the intrinsic ball of radius $\frac{1}{2c_\sharp H(p)}$ about $p$.
\end{lemma}
\begin{proof}
For any unit speed geodesic $\gamma:[0,s]\to \M$ joining the points $y=\gamma(0)$ and $x=\gamma(s)$, we have
\[
\cd_{\gamma'}H^{-1}\leq c_\sharp\,.
\]
Integrating yields
\[
-c_\sharp s\leq H^{-1}(x)-H^{-1}(y)\leq c_\sharp s
\]
or, if $s\leq \frac{1}{2c_\sharp H(x)}$, 
\[
\frac{H(x)}{2}\leq\frac{H(x)}{1+c_\sharp H(x)s}\leq H(y)\leq\frac{H(x)}{1-c_\sharp H(x)s}\leq 2H(x)\,.
\]
The claim follows.
\end{proof}
%In particular, if we translate and rescale so that the mean curvature is $1$ at the origin, then it is bounded by $2$ in an intrinsic ball of radius $(2C)^{-1}$ about the origin.

\begin{comment}
A slightly more involved argument yields the following Lemma.
\begin{lemma}\label{lem:consequence of gradient estimate}
Let $X:\M\to\mathbb{S}_K^{n+1}$ be a properly immersed hypersurface. Suppose that
\[
H(p)\ge h_\sharp K^{\frac{1}{2}}\;\implies\; \vert \cd H(p)\vert\le c_\sharp H^2(p)
\]
for some choice of constants $h_\sharp$ and $c_\sharp$. If $H(p_0)>\gamma h_\sharp K^{\frac{1}{2}}$, then
\[
H(p)\ge\gamma^{-1}H(p_0)\;\;\text{for all}\;\; p\in \mathcal{B}_{\frac{\gamma-1}{c_\sharp H(p_0)}}(p_0)\,.
\]
%where $\mathcal{B}_r(p)$ denotes the intrinsic ball of radius $r$ about $p$.
\end{lemma}
\begin{proof}
The proof is identical to that of \cite[Lemma 6.6]{HuSi09}.
\end{proof}
\end{comment}

\subsection{Higher order estimates} 
 
The gradient estimate can be used to bound the first order terms which arise in the evolution equation for $\cd^2{A}$. A straightforward maximum principle argument exploiting this observation yields an analagous estimate for $\cd^2{A}$.

\begin{theorem}[Hessian estimate (cf. \cite{Hu84,HuSi09})]\label{thm:Hessian estimate}
Let $X:\M\times[0,T)\to\mathbb{S}_K^{n+1}$, $n\geq 4$, be a solution to mean curvature flow with initial condition in the class $\mathcal{C}_K^n(\alpha,V,\Theta)$. There exists $C=C(n,\alpha,V,\Theta)$ such that
\ba\label{eq:scale_invariant_Hessian_estimate}
\vert\cd^2{A}\vert^2\leq C(H^{6}+K^{3})\;\;\text{in}\;\; \M \times[\lambda_0K^{-1},T)\,.
\ea
\end{theorem}
\begin{proof}
We proceed as in \cite[Theorem 6.3]{HuSi09}. By \eqref{eqn_evolderiv2},
\[
(\pd_t-\Delta)\vert\cd^2{A}\vert^2\leq c\left(W\vert\cd^2{A}\vert^2+\vert\cd{A}\vert^4\right)-2\vert\cd^3{A}\vert^2\,,
\]
where $c$ depends only on $n$. It follows that
\bann
(\pd_t-\Delta)\frac{\vert\cd^2{A}\vert^2}{W^\frac{5}{2}}\leq{}& \frac{c}{W^\frac{5}{2}}\lsb W\vert\cd^2{A}\vert^2+\vert\cd{A}\vert^4\rsb-2\frac{\vert\cd^3{A}\vert^2}{W^\frac{5}{2}}\\
{}& \! -5a\frac{\vert\cd^2{A}\vert^2}{W^\frac{7}{2}}\lsb(\vert{A}\vert^2+nK)H^2-\vert\cd H\vert^2\rsb\\
{}&+\frac{5}{W^\frac{7}{2}}\inner{\cd\vert\cd^2{A}\vert^2}{\cd W}-\frac{25}{2}\frac{\vert\cd^2A\vert^2}{W^{\frac{7}{2}}}\frac{\vert \cd W\vert^2}{W}.
\eann

We can use the good third order term on the first line to absorb the penultimate term, since
\bann
\frac{5}{W^\frac{7}{2}}\inner{\cd\vert\cd^2{A}\vert^2}{\cd W}\le{}&\frac{10}{W^\frac{7}{2}}\vert\cd^3{A}\vert\vert\cd^2A\vert\vert\cd W\vert\\
\le {}&\frac{1}{W^\frac{1}{2}}\left(\frac{\vert\cd^3{A}\vert^2}{W^2}+25\frac{\vert\cd^2A\vert^2\vert\cd W\vert^2}{W^4}\right).
\eann
Estimating
\[
\frac{\vert\cd W\vert^2}{W}\le 4a\vert\cd H\vert^2
\]
then yields
\bann
(\pd_t-\Delta)\frac{\vert\cd^2{A}\vert^2}{W^\frac{5}{2}}\leq{}& \frac{c}{W^\frac{5}{2}}\lsb W\vert\cd^2{A}\vert^2+\vert\cd{A}\vert^4\rsb-\frac{\vert\cd^3{A}\vert^2}{W^\frac{5}{2}}
%{}&+5a\frac{\vert\cd^2{A}\vert^2}{W^\frac{7}{2}}\vert\cd H\vert^2+\frac{25}{2}\frac{\vert\cd^2A\vert^2}{W^{\frac{7}{2}}}\frac{\vert \cd W\vert^2}{W}.\\
+55a\frac{\vert\cd^2{A}\vert^2}{W^\frac{7}{2}}\vert\cd H\vert^2.
\eann

Estimating the first order terms using Theorem \ref{thm:gradient estimate} then yields
\bann
(\pd_t-\Delta)\frac{\vert\cd^2{A}\vert^2}{W^{\frac{5}{2}}}\leq{}&c_1\frac{\vert\cd^2{A}\vert^2}{W^{\frac{3}{2}}}+C_1K^2\frac{\vert\cd^2{A}\vert^2}{W^{\frac{7}{2}}}\mathrm{e}^{-\delta Kt}\\
{}&+\frac{c_1H^4W^2+C_1K^4\mathrm{e}^{-2\delta Kt}}{W^{\frac{5}{2}}}-\frac{\vert\cd^3{A}\vert^2}{W^{\frac{5}{2}}}\,,
\eann
where $c_1$ depends only on $n,\alpha$ and $\Theta$, and $C_1$ depends also on $V$.

Similar arguments yield
\bann
(\pd_t-\Delta)\frac{\vert\cd{A}\vert^2}{W^{\frac{3}{2}}}%={}&\frac{(\pd_t-\Delta)\vert\cd{A}\vert^2}{W^{\frac{3}{2}}}-\frac{\vert\cd{A}\vert^2}{W^3}(\pd_t-\Delta)W^{\frac{3}{2}}+2\inner{\cd\frac{\vert\cd A\vert^2}{W^{\frac{3}{2}}}}{\frac{\cd W^{\frac{3}{2}}}{W^{\frac{3}{2}}}}\\
%={}&\frac{(\pd_t-\Delta)\vert\cd{A}\vert^2}{W^{\frac{3}{2}}}-\frac{3}{2}\frac{\vert\cd{A}\vert^2}{W^\frac{5}{2}}(\pd_t-\Delta)W+\frac{3}{4}\frac{\vert\cd{A}\vert^2}{W^{\frac{5}{2}}}\frac{\vert \cd W\vert^2}{W}\\
%{}&+3\inner{\frac{\cd\vert\cd A\vert^2}{W^{\frac{3}{2}}}-\frac{3}{2}\frac{\vert\cd A\vert^2}{W^{\frac{5}{2}}}\cd W}{\frac{\cd W}{W}}\\
%\le{}&\frac{cW\vert\cd A\vert^2-2\vert\cd^2A\vert^2}{W^{\frac{3}{2}}}-3a\frac{\vert\cd{A}\vert^2}{W^\frac{5}{2}}\lsb(\vert A\vert^2+nK)H^2-\vert\cd H\vert^2\rsb\\
%{}&-\frac{15}{4}\frac{\vert\cd{A}\vert^2}{W^{\frac{5}{2}}}\frac{\vert \cd W\vert^2}{W}+6\frac{\vert\cd^2 A\vert\vert\cd A\vert\vert\cd W\vert}{W^{\frac{5}{2}}}\\
%\le{}&c\frac{\vert\cd A\vert^2}{W^{\frac{5}{2}}}\left(W^2+\vert\cd H\vert^2\right)-\frac{\vert \cd^2A\vert^2}{W^{\frac{3}{2}}}\\
\leq{}&\frac{c_2H^2W^3+C_2K^4\mathrm{e}^{-2\delta Kt}}{W^{\frac{5}{2}}}-\frac{\vert\cd^2{A}\vert^2}{W^\frac{3}{2}}\,,
\eann
and
\bann
(\pd_t-\Delta)\frac{\vert\cd{A}\vert^2}{W^{\frac{7}{2}}}%={}&\frac{(\pd_t-\Delta)\vert\cd{A}\vert^2}{W^{\frac{7}{2}}}-\frac{\vert\cd{A}\vert^2}{W^7}(\pd_t-\Delta)W^{\frac{7}{2}}+2\inner{\cd\frac{\vert\cd A\vert^2}{W^{\frac{7}{2}}}}{\frac{\cd W^{\frac{7}{2}}}{W^{\frac{7}{2}}}}\\
%={}&\frac{(\pd_t-\Delta)\vert\cd{A}\vert^2}{W^{\frac{7}{2}}}-\frac{7}{2}\frac{\vert\cd{A}\vert^2}{W^\frac{9}{2}}(\pd_t-\Delta)W+\frac{35}{4}\frac{\vert\cd{A}\vert^2}{W^{\frac{9}{2}}}\frac{\vert \cd W\vert^2}{W}\\
%{}&+7\inner{\frac{\cd\vert\cd A\vert^2}{W^{\frac{7}{2}}}-\frac{7}{2}\frac{\vert\cd A\vert^2}{W^{\frac{9}{2}}}\cd W}{\frac{\cd W}{W}}\\
%\le{}&\frac{cW\vert\cd A\vert^2-2\vert\cd^2A\vert^2}{W^{\frac{7}{2}}}-7a\frac{\vert\cd{A}\vert^2}{W^\frac{9}{2}}\lsb(\vert A\vert^2+nK)H^2-\vert\cd H\vert^2\rsb\\
%{}&-\frac{63}{4}\frac{\vert\cd{A}\vert^2}{W^{\frac{9}{2}}}\frac{\vert \cd W\vert^2}{W}+14\frac{\vert\cd^2 A\vert\vert\cd A\vert\vert\cd W\vert}{W^{\frac{9}{2}}}\\
\le{}&c\frac{\vert\cd A\vert^2}{W^{\frac{9}{2}}}\left(W^2+\vert\cd H\vert^2\right)-\frac{\vert \cd^2A\vert^2}{W^{\frac{7}{2}}}\\
\leq{}&\frac{c_3H^2W^3+C_3K^4\mathrm{e}^{-2\delta Kt}}{W^{\frac{9}{2}}}-\frac{\vert\cd^2{A}\vert^2}{W^\frac{7}{2}}\,,
\eann
where $c_2$ and $c_3$ depend only on $n$, $\alpha$, and $\Theta$, and $C_2$ and $C_3$ depend also on $V$.

Setting
\[
f:= \frac{\vert\cd^2{A}\vert^2}{W^{\frac{5}{2}}}+c_1\frac{\vert\cd{A}\vert^2}{W^{\frac{3}{2}}} +C_1K^2\frac{\vert\cd{A}\vert^2}{W^{\frac{7}{2}}}
\]
and estimating $W\ge K$, we obtain
\bann
(\pd_t-\Delta)f\leq{}&\frac{c_1H^4W^2+C_1K^4\mathrm{e}^{-2\delta Kt}}{W^{\frac{5}{2}}}+c_1\frac{c_2H^2W^3+C_2K^4\mathrm{e}^{-2\delta Kt}}{W^{\frac{5}{2}}}\\
{}&+C_1K^2\frac{c_3H^2W^3+C_3K^4\mathrm{e}^{-2\delta Kt}}{W^{\frac{9}{2}}}\\
\le{}&\frac{(c_1a^2+c_1c_2+c_3C_1)H^2W^3+(C_1+c_1C_2+C_1C_3)K^4\mathrm{e}^{-2\delta Kt}}{W^{\frac{5}{2}}}\\
\le{}&(c_1a^2+c_1c_2+c_3C_1)H^2W^{\frac{1}{2}}+(C_1+c_1C_2+C_1C_3)K^\frac{3}{2}\mathrm{e}^{-2\delta Kt}\\
\le{}&c_4(\vert A\vert^2+nK)H+C_4K^\frac{3}{2}\mathrm{e}^{-2\delta Kt}\,.
\eann
Thus,
\bann
(\pd_t-\Delta)\lb f-c_4H+\frac{C_4}{2\delta}K^{\frac{1}{2}}\mathrm{e}^{-2\delta Kt}\rb\leq{}&0\,.
\eann
The maximum principle and Proposition \ref{prop:class C universal interior estimates} then yield
\bann
\max_{\M\times\{t\}}(f-c_4H)\leq{}&\max_{\M\times\{\lambda_0K^{-1}\}}(f-c_4H)+\frac{C_4}{2\delta}K^{\frac{1}{2}}\lb\mathrm{e}^{-2\delta \lambda_0}-\mathrm{e}^{-2\delta Kt}\rb\\
\leq{}&C_5K^{\frac{1}{2}}
\eann
for all $t\geq \lambda_0K^{-1}$, where $C_5$ depends only on $n$, $\alpha$, $V$, and $\Theta$. We conclude that
\[
\vert\cd^2{A}\vert^2\leq cHW^{\frac{5}{2}}+CK^{\frac{1}{2}}W^{\frac{5}{2}}\;\;\text{in}\;\; \M\times [\lambda_0K^{-1},T)\,,
\]
where $c$ and $C$ depend only on $n$, $\alpha$, $V$, and $\Theta$. The claim now follows from Young's inequality.
\end{proof}

Applying the Hessian estimate in conjunction with the the evolution equation \eqref{eq_sff} for $A$ yields an analogous bound for $\cd_t A$, and hence, in particular, for the time derivative of $H$. Thus, in high curvature regions, we obtain the following a priori bounds for $\cd H$ and $\pd_t H$.

\begin{corollary}\label{cor:spacetime grad H bound}
Let $X:\M\times[0,T)\to\mathbb{S}_K^{n+1}$, $n\geq 2$, be a solution to mean curvature flow with initial condition in the class $\mathcal{C}_K^n(\alpha,V,\Theta)$. There exist $h_\sharp=h_\sharp(n,\alpha,V,\Theta)$ and $c_\sharp=c_\sharp(n,\alpha,V,\Theta)$ such that
\begin{equation}
H(x,t)\geq h_\sharp \sqrt K
\;\implies \;
%\;\; \text{implies} \;\;
\frac{\vert \cd H\vert}{H^2}(x,t)\leq c_\sharp\;\;\text{and}\;\;\frac{\vert \pd_tH\vert}{H^3}(x,t)\leq \frac{c_\sharp^2}{2}.
\end{equation}
\end{corollary}

This is a very useful estimate in light of the following `parabolic' version of Lemma \ref{lem:consequence of gradient estimate}.

\begin{lemma}\label{lem:integrate the gradient estimate (parabolic)}
Let $X:\M\times[0,T)\to\mathbb S_K^{n+1}$ be a solution to mean curvature flow. If, given $c_\sharp<\infty$,
\[
\frac{\vert \cd H\vert}{H^2}\leq c_\sharp\;\;\text{ and }\;\;\frac{\vert \pd_tH\vert}{H^3}\leq \frac{c_\sharp^2}{2}\,,
\]
then
\begin{equation}
\frac{H(p,t)}{10}\leq H(q,s)\leq 10H(p,t)
\end{equation}
for all $(q,s)\in \mathcal{P}_{\frac{1}{10c_\sharp H(p,t)}}(p,t)$, the intrinsic parabolic cylinder in $\M\times [0,T)$ of radius $\frac{1}{10c_\sharp H(p,t)}$ about $(p,t)$.
\end{lemma}
\begin{proof}
Fix $\gamma\in[\frac{1}{2},1)$. As in Lemma \ref{lem:consequence of gradient estimate}, given any $r\leq\frac{1-\gamma}{c_\sharp H(p,t)}$,
\bann%\label{eq:H bound at time t}
\gamma H(p,t)\leq \frac{H(p,t)}{1+c_\sharp H(p,t)r}\leq H(q,t)\leq \frac{H(p,t)}{1-c_\sharp H(p,t)r}\leq \gamma^{-1}H(p,t)
\eann
for all $q\in \mathcal{B}_{r}(p,t)$, the $g_t$-intrinsic ball of radius $r$ about the point $p$. Given $q\in \mathcal{B}_{r}(p,t)$, set $h(t):= H(p,t)$. Then
\[
-c_\sharp^2\leq (h^{-2})'(s)\leq c_\sharp^2\,.
\]
Since $r\leq \frac{1-\gamma}{c_\sharp H(p,t)}\leq \frac{\gamma}{c_\sharp H(p,t)}\leq \frac{1}{c_\sharp H(q,t)}$, integrating between $s\in (t-r^2,t]$ and $t$ yields
\[
\frac{H(q,t)}{\sqrt{1-c_\sharp^2H^2(q,t)r^2}}\leq H(q,s)\leq \frac{H(q,t)}{\sqrt{1-c_\sharp^2H^2(q,t)r^2}}
\]
and hence
\[
\frac{H(p,t)}{\sqrt{\gamma^{-2}+c_\sharp^2H^2(p,t)r^2}}\leq H(q,s)\leq \frac{H(p,t)}{\sqrt{\gamma^2-c_\sharp^2H^2(p,t)r^2}}
\]
for all $(q,s)\in \mathcal{P}_r(p,t)$, so long as $r\leq \frac{\gamma}{c_\sharp H(p,t)}$. The claim follows upon choosing, say, $\gamma=1/2$.
\end{proof}

An inductive argument, exploiting estimates for lower order terms in the evolution equations for higher derivatives of $A$ as in Theorem \ref{thm:Hessian estimate}, can be applied to obtain estimates for spatial derivatives of $A$ to all orders. The evolution equation for $A$ then yields bounds for the mixed space-time derivatives (cf. \cite[Theorem 6.3 and Corollary 6.4]{HuSi09}). We state these estimates here, however they will not actually be needed in the construction of the surgically modified flows.

\begin{theorem}[Higher-order estimates]\label{thm:higher order estimates}
Let $X:\M\times[0,T)\to\mathbb{S}_K^{n+1}$, $n\geq 4$, be a solution to mean curvature flow with initial condition in the class $\mathcal{C}_K^n(\alpha,V,\Theta)$. There exist, for each pair of non-negative integers $k$ and $\ell$, constants $C_{k,\ell}=C_{k,\ell}(k,\ell,n,\alpha,V,\Theta)$ such that
\ba\label{eq:scale_invariant_higher_derivative_estimates}
\vert\cd_t^k\cd^\ell {A}\vert^2\leq C_{k,\ell}(H^{2+4k+2\ell}+K^{1+2k+\ell})\;\;\text{in}\;\; \M \times[\lambda_0K^{-1}\!,T).
\ea
\end{theorem}

\subsection{Neck detection}

The cylindrical and gradient estimates imply that, in regions of very high curvature, solutions either form high quality `neck' regions, or else become locally uniformly convex. 

\begin{lemma}[Curvature necks (cf. {\cite[Lemma 7.4]{HuSi09}})]\label{lem:curvature neck detection}
Let $X:\M\times[0,T)\to\mathbb{S}_K^{n+1}$ be a solution to mean curvature flow with initial condition in the class $\mathcal{C}^n_K(\alpha,V,\Theta)$. Given $\varepsilon\leq \frac{1}{100}$, there exist parameters $\eta_\sharp=\eta_\sharp(n,\alpha,V,\Theta,\varepsilon)>0$ and $h_\sharp=h_\sharp(n,\alpha,V,\Theta,\varepsilon)<\infty$ with the following property. If
\bann
H(p_0,t_0)\ge h_\sharp\sqrt K\;\;\text{ and }\;\; \lambda_1(p_0,t_0)\leq \eta_\sharp H(p_0,t_0)\,,
\eann
then %$(x_0,t_0)$ lies at the center of a \emph{shrinking curvature neck of quality $\varepsilon$}. That is,
\[
\Lambda_{r_0,k,\varepsilon}(p_0,t_0)\leq \varepsilon r_0^{k+1}%\left(r_0^2+2(n-m)(t_0-t)\right)^{\frac{k+1}{2}}
\]
for each $k=0,\dots,\lfloor\frac{2}{\varepsilon}\rfloor$, where $r_0:= \frac{n-1}{H(p_0,t_0)}$,
\[
\Lambda_{r,0,\varepsilon}(p,t):= \max_{\mathcal{B}_{\varepsilon^{-1}r}(p,t)\times(t-10^4r^2,t]}\sqrt{\lambda_1^2+\sum_{j=2}^n(\lambda_n-\lambda_j)^2}\,,
\]
and, for each $k\geq 1$,
\[
\Lambda_{r,k,\varepsilon}(p,t):= \max_{\mathcal{B}_{\varepsilon^{-1}r}(p,t)\times(t-10^4r^2,t]}\vert \cd^k{A}\vert\,.
\]
\end{lemma}
\begin{proof}
The proof is essentially that of \cite[Lemma 7.4]{HuSi09}. %We incclude the proof here

%\textcolor{red}{\texttt{So we can leave it out if needed}.}

Suppose that the claim does not hold. Then for some $n\geq 3$ there must exist parameters $\alpha$, $V$ and $\Theta$, some $\varepsilon_0<\frac{1}{100}$, a sequence of solutions $X_j:\M_j\times[0,T_j)\to\mathbb{S}_{K_j}^{n+1}$ to mean curvature flow with $X_j(\vts\cdot\vts,0)\in\mathcal{C}^n_{K_j}(\alpha,V,\Theta)$, and points $(p_j,t_j)\in \M_j^n\times[0,T_j)$ such that
\[
%(n-m)r_j^{-1}:= 
H_j(p_j,t_j)\geq (n-1)j\sqrt{K_j}\;\;\text{and}\;\; \frac{\lambda^j_1}{H_j}(p_j,t_j)\leq j^{-1}\,,
\]
and yet
\begin{equation}\label{eq:curvature neck contra1}
\Lambda^j_{r_j,k_j,\varepsilon_0}(p_j,t_j)\geq \varepsilon_0r_j^{k+1}%\left(r_j^2+2(n-m)(t_j-t)\right)^{\frac{k+1}{2}}
\end{equation}
for some $k_j\leq \lfloor\varepsilon_0^{-1}\rfloor$ for each $j$, where $r_j:=\frac{n-1}{H_j(x_j,t_j)}$ and we denote objects defined along $X_j$ using the a sub- or superscript $j$. After passing to a subsequence, we may arrange that \eqref{eq:curvature neck contra1} holds for some fixed $k_j=k_0\leq \lfloor\varepsilon_0^{-1}\rfloor$ for all $j$. After translating the points $(X_j(x_j,t_j),t_j)$ to the space-time origin in $\R^{n+2}\times\R$ and rotating so that the tangent plane to the sphere at the origin is $\R^{n+1}\times\{0\}$ with upward pointing normal, and parabolically rescaling by $r_j$, we obtain a sequence of flows
\[
\hat X_j:\M_j^n\times [-r_j^{-2}t_j,r_j^{-2}(T_j-t_j))\to\mathbb{S}_{r_j^2K_j}^{n+1}-r_j^{-1}K_j^{-\frac{1}{2}}e_{n+2}
\]
given by
\begin{equation}
\hat X_j(x,t):= r_j^{-1}O_j(X_j(x,r_j^{2}t+t_j)-X_j(x_j,t_j))\,,
\end{equation}
where $O_j\in \mathrm{SO}(n+1)$. Each $\hat X_j$ is in the class $\mathcal{C}^n_{r_j^2K_j}(\alpha,V,\Theta)$ and satisfies $\hat X_j(x_j,0)=0$,
\begin{equation}\label{eq:curvature of hat X at origin}
\hat H_j(x_j,0)=n-1\,,\;\text{ and }\;\; \frac{\hat\lambda^j_1}{\hat H_j}(x_j,0)\leq j^{-1} ,
\end{equation}
but
\begin{equation}\label{eq:curvature neck contra2}
\hat \Lambda^j_{1,k_0,\varepsilon_0}(x_j,0)\geq \varepsilon_0\,,
\end{equation}
where we denote objects defined along $\hat X_j$ using a $\hat{(\,\cdot\,)}$ and the sub- or superscript $j$. We claim that the new sequence subconverges locally uniformly in the smooth topology to a shrinking cylinder solution in the Euclidean space $\R^{n+1}\times\{0\}$, in contradiction with \eqref{eq:curvature neck contra2}. First note that, by \eqref{eq:class C universal T bound}, $K_jt_j\ge C(n,\alpha,\Theta)>0$ and hence $-r_j^{-2}t_j\to-\infty$ as $j\to\infty$. We claim that the mean curvature of $\hat X_j$ is uniformly bounded on an intrinsic parabolic cylinder of uniform radius about $(x_j,0)$, so long as $j$ is sufficiently large. Indeed, by Theorems \ref{thm:gradient estimate} and \ref{thm:Hessian estimate}, we can find constants $c_\sharp$ (depending only on $n$, $\alpha$ and $\Theta$) and $C$ (depending only on $n$, $\alpha$, $V$ and $\Theta$) such that
\[
\vert\hat\cd_j\hat H_j\vert\leq \frac{c_\sharp}{2}\hat H_j^2+Cj^{-2}\;\;\text{and}\;\; \vert\pd_t\hat H_j\vert\leq \frac{c_\sharp^2}{4}\hat H_j^3+Cj^{-3}
\]
in $\M_j\times [-r_j^{-2}t_j+j^{-2}/4,0]$. Thus, given any $\rho>0$, we can find $j_0\in \mathbb{N}$ such that
\[
\vert\hat\cd_j\hat H_j\vert\leq c_\sharp\hat H_j^2\;\;\text{and}\;\; \vert\pd_t\hat H_j\vert\leq \frac{c_\sharp^2}{2}\hat H_j^3
\]
in $\M^n_j\times [-\rho,0]$ for $j\geq j_0$. Lemma \ref{lem:integrate the gradient estimate (parabolic)} now implies that
\[
\frac{n-1}{10}=\frac{\hat H_j(x_j,0)}{10}\leq \hat H_j(y,s)\leq 10\hat H_j(x_j,0)=10(n-1)
\]
for any $(y,s)\in \mathcal{P}^j_{\frac{1}{10c_\sharp}}(x_j,0)$ for all sufficiently large $j$. It follows that some subsequence of the restricted mean curvature flows $\hat X_j|_{\mathcal{P}^j_{\frac{1}{10 c_\sharp}}(x_j,0)}$ converges locally uniformly in the smooth topology to a limiting mean curvature flow $\hat X:U\times(-\frac{1}{100c_\sharp^2},0]\to\R^{n+1}\times\{0\}$ (which may not be proper). We claim that the limit flow is part of a shrinking cylinder. We shall denote objects defined along the limit using a $\hat{(\,\cdot\,)}$. Indeed, by the cylindrical estimate (Theorem \ref{thm:cylindrical estimate}), $\hat X$ satisfies
\ba\label{eq:hat X cylindrical}
\vert \hat A\vert^2-\frac{1}{n-1}\hat H^2\le 0\,.
\ea
In particular, $\hat\lambda_1$ is non-negative. On the other hand, by \eqref{eq:curvature of hat X at origin}, $\hat \lambda_1$ vanishes at the origin. Thus, by the splitting theorem, $\hat X$ splits locally off a line. But then \eqref{eq:hat X cylindrical} implies that the cross section of the splitting is umbilic. We need to extend the convergence to a sufficiently large region. This can be achieved since, a posteriori, the mean curvature could not have increased very much in $\mathcal{P}^j_{\frac{1}{10c_\sharp}}$ (due to the convergence to a shrinking cylinder solution). That is,
\[
\hat H_j(y,s)\leq 2(n-1)
\]
for all $(y,s)\in \mathcal{P}_{\frac{1}{10c_\sharp}}(x_j,0)$ so long as $j$ is sufficiently large. Applying the gradient estimates as before, we obtain uniform bounds for $\hat H_j$ on the uniformly larger neighborhood $\mathcal{P}^j_{\frac{2}{10c_\sharp}}(x_j,0)$. Repeating the previous argument, we conclude that a subsequence of the flows $\hat X_j|_{\mathcal{P}^j_{\frac{2}{10c_\sharp}}(x_j,0)}$ converge to a part of shrinking cylinder. After repeating the argument a finite number of times, we obtain convergence of a subsequence of the flows $\hat X_j|_{\mathcal{P}^j_{2\varepsilon_0^{-1}}(x_j,0)}$ to a part of a shrinking cylinder. Since the convergence is smooth on compact subsets of spacetime, this violates \eqref{eq:curvature neck contra2}.
\end{proof}

\begin{definition}
Let $X:\M\to\mathbb{S}_K^{n+1}\subset \R^{n+2}$ be an immersed hypersurface of $\mathbb{S}_K^{n+1}$. A point $p\in \M$ \emph{lies at the center of an} $(\varepsilon,k,L)$-\emph{neck of size $r$} if the map $\exp_{r^{-1}X(p)}^{-1}\circ (r^{-1}X)$ is $\varepsilon$-cylindrical and $(\varepsilon,k)$-parallel at all points in the induced intrinsic ball of radius $L$ about $p$ in the sense of \cite[Definition 3.9]{HuSi09}.
\end{definition}

By \cite[Propositions 3.4 and 3.5]{HuSi09}, these ``curvature'' necks can be integrated to obtain ``hypersurface'' necks in the tangent space, which can be replaced by a pair of ``convex caps'' in a controlled way (see \cite[Section 3]{HuSi09}).

\section{The key estimates for surgically modified flows}\label{sec:key estimates surgery}

We need to show that suitable versions of the key estimates still hold in the presence of surgeries. In the following definition, surgery is performed on the middle third of a neck of size $r$ in the obvious way: 
\begin{enumerate}[(i)]
\item First scale by $r^{-1}$ and precompose with $\exp_{r^{-1}X(p)}^{-1}$ to obtain a neck in $T_{r^{-1}X(p)}\mathbb S_{r^2 K}^{n+1}$, 
\item Perform the surgery on the middle third of this neck in $T_{r^{-1}X(p)}\mathbb S_{r^2K}^{n+1}$ as described in \cite[Section 3]{HuSi09}, 
\item Re-embed in $\mathbb S_{K}^{n+1}$ by composing with $\exp_{r^{-1}X(p)}$ and scaling by $r$.
\end{enumerate}
\begin{definition}
A \emph{surgically modified (mean curvature) flow} in $\mathbb S_K^{n+1}$ with neck parameters $(\varepsilon,k,L)$, surgery parameters $(\tau,B)$, and surgery scale $r$ is a finite sequence $\{X_i:\M_i^n\times[T_i,T_{i+1}]\to\mathbb{S}_K^{n+1}\}_{i=1}^{N-1}$ of smooth mean curvature flows $X_i:\M_i^n\times[T_i,T_{i+1}]\to\mathbb{S}_K^{n+1}$ for which the $(i+1)$-st initial datum $X_{i+1}(\cdot,T_{i+1}):\M_{i+1}\to\mathbb S_K^{n+1}$ is obtained from the $i$-th final datum $X_{i}(\cdot,T_{i+1}):\M_{i}\to\mathbb S_K^{n+1}$ by performing finitely many $(\tau,B)$-standard surgeries, in the sense of \cite[Section 3]{HuSi09}, on the middle thirds of $(\varepsilon,k,L)$-necks with mean curvature satisfying $\frac{n-1}{10r} \le H\le \frac{10(n-1)}{r}$, and then discarding finitely many connected components that are diffeomorphic either to $\mathbb S^n$ or to $\mathbb S^1\times \mathbb S^{n-1}$.
\end{definition}

\subsection{Quadratic and inscribed/exscribed curvature pinching}\label{ssec:pinching under surgery}

For a suitable range of neck and surgery parameters, and surgery scales, the surgery procedure of \cite[Section 3]{HuSi09} preserves the quadratic pinching condition \eqref{eq:strict quadratic pinching}. Indeed, the surgery replaces a nearly cylindrical Euclidean neck satisfying $|A|^2 \simeq \frac{1}{n-1}H^2$ with a pair of Euclidean convex caps satisfying $|A|^2 \simeq \frac{1}{n}H^2$ (these estimates are carried out precisely in \cite[Corollary 3.20]{Nguyen2020}). Since the surgery scale may be taken arbitrarily small, the same can be ensured after re-embedding in $\mathbb S_K^{n+1}$.

When $X_0:\M\to\mathbb S_K^{n+1}$ is an embedding, we can also preserve the inscribed curvature pinching
\[
\max_{\M\times\{0\}}\frac{\overline k}{F}\le \mu_0
\]
for any constant $\mu_0\ge \sqrt{\frac{(n-2)(n-2+\alpha)}{4\alpha}}$. Indeed, by Proposition \ref{prop:noncollapsing}, $\max_{\M\times\{0\}}\frac{\overline k}{F}$ does not decay between surgeries. Moreover, using \cite[Theorem 3.26]{HuSi09}, we can arrange, for suitable neck and surgery parameters, and surgery scales, that
\[
\frac{\overline k}{F}\le \sqrt{\frac{(n-2)(n-2+\alpha)}{4\alpha}}
\]
on the regions modified or added by surgery. A similar argument applies to the exscribed curvature. %In fact, the estimate improves since the elapsed time between any pair of surgery times is bounded uniformly from below (by Proposition \ref{prop:class C universal interior estimates}) and we have a good decay term in \eqref{eq:noncollapsing}.

\subsection{The cylindrical estimate}

We first note that the function $(f_{\sigma,\eta})_+$ is pointwise non-increasing in regions modified by surgery.

\begin{lemma}\label{lem:f_+ nonincreasing under surgery}
Given $n\ge 3$ and $K>0$, there exist parameters $\eta_0>0$, $\sigma_0\in(0,1)$, neck parameters $\varepsilon_0>0$, $k_0\geq 2$, surgery parameters $\tau$, $B$, and a surgery scale $r_0>0$ such that, for any $\sigma\in(0,\sigma_0]$ and $\eta\in(0,\eta_0]$, the function $(f_{\sigma,\eta})_+$ is
\begin{itemize}
\item zero on regions added by, and 
\item non-increasing on regions modified by
\end{itemize}
standard surgery with parameters $\tau_0$, $B$ on an $(\varepsilon,k,L)$-neck with mean curvature satisfying $H\ge \frac{(n-1)}{10r}$ for any $\varepsilon\in(0,\varepsilon_0]$, $k\geq k_0$, $L\ge 10$ and $r\in(0,r_0]$.
\end{lemma}
\begin{proof}
This follows readily from \cite[Proposition 4.5]{HuSi09} and the Gauss equation.
\end{proof}

In the following theorem, we assume that the parameters $\eta$, $\sigma$, the neck parameters $\varepsilon$, $k$, $L$, the surgery parameters $\tau$, $B$, and the surgery scale $r$ are chosen within the range for which Lemma \ref{lem:f_+ nonincreasing under surgery} applies. 

%\textcolor{red}{\texttt{This will simplify the statements of the following estimates}}.

\begin{comment}
\begin{theorem}[Cylindrical estimate for surgically modified flows (Cf. {\cite[Theorem 5.3]{HuSi09}})]\label{thm:cylindrical surgery}
For all $n\ge 3$ and $K>0$, there exist $\eta_0>0$, neck parameters $\varepsilon_0>0$ and $k_0\ge 2$, and surgery parameters $\tau_0$, $B$ with the following property. Let $\{X_i:\M_i^n\times[T_i,T_{i+1}]\to\mathbb{S}_K^{n+1}\}_{i=1}^{N-1}$ be a surgically modified flow with initial condition in the class $\mathcal{C}_K^n(\alpha,V,\Theta)$ (with $\alpha>\frac{2}{3}$ when $n=3$) and $(\tau_0,B)$-standard surgeries on $(\varepsilon,k)$-hypersurface necks with $\varepsilon<\varepsilon_0$ and $k\ge k_0$. For every $\eta\in(0,\eta_0)$ there exists $C_\eta=C_\eta(n,\alpha,V,\Theta,\eta)<\infty$ such that
\begin{align}\label{eq:cylindrical estimate surgery}
|A|^2-\frac{1}{n-1}H^2 \leq \eta H^2 + C_\eta K \quad\text{in}\quad \mc M_i^n\times[T_i,T_{i+1}]
\end{align}
for all $i$.
\end{theorem}
\end{comment}
\begin{theorem}[Cylindrical estimate for surgically modified flows (Cf. {\cite[Theorem 5.3]{HuSi09}})]\label{thm:cylindrical surgery}
Let $\{X_i:\M_i^n\times[T_i,T_{i+1}]\to\mathbb{S}_K^{n+1}\}_{i=1}^{N-1}$, $n\ge 3$, be a surgically modified flow with initial condition in the class $\mathcal{C}_K^n(\alpha,V,\Theta)$ (with $\alpha>\frac{2}{3}$ when $n=3$). For every $\eta\in(0,\eta_0)$ there exists $C_\eta=C_\eta(n,\alpha,V,\Theta,\eta)<\infty$ such that
\begin{align}\label{eq:cylindrical estimate surgery}
|A|^2-\frac{1}{n-1}H^2 \leq \eta H^2 + C_\eta K \quad\text{in}\quad \mc M_i\times[T_i,T_{i+1}]
\end{align}
for all $i$.
\end{theorem}
\begin{proof}
Proceeding as in the proof of Theorem \ref{thm:cylindrical estimate} but with $\delta$ taken to be zero, we obtain an analogue of \eqref{eq:v_k monotonicity} on each time interval $(T_i,T_{i+1})$, with $v_k$ replaced by $(f_{\sigma,\eta}-k)_+^{\frac{p}{2}}$. By Lemma \ref{lem:f_+ nonincreasing under surgery}, this can be integrated from $T_1=0$ to $T_N=T$ to obtain an analogue of \eqref{eq:still holds with surgeries}. The remainder of the proof of the cylindrical estimate then applies unmodified.
\end{proof}

%In fact, the constant in the estimate improves since the elapsed time between any pair of surgery times is bounded uniformly from below (by Proposition \ref{prop:class C universal interior estimates}) and we have a good decay term in \eqref{eq:cylindrical estimate}.

Henceforth, when we refer to a surgically modified flow, we will assume that the neck and surgery parameters, and the surgery scale, are fixed within a suitable range, which we progressively refine.

\subsection{The gradient estimate} 

\begin{comment}
\begin{theorem}[Gradient estimate for surgically modified flows (Cf. {\cite[Theorem 6.1]{HuSi09}})]\label{thm:gradient surgery}
For all $n\ge 3$, $K>0$ and $\Lambda>0$, there exist $\eta_0>0$, neck parameters $\varepsilon_0>0$ and $k_0\ge 2$, and surgery parameters $(\tau_0,B)$ with the following property. Let $\{X_i:\M_i^n\times[T_i,T_{i+1}]\to\mathbb{S}_K^{n+1}\}_{i=1}^{N-1}$ be a surgically modified flow with initial condition in the class $\mathcal{C}_K^n(\alpha,V,\Theta)$ (with $\alpha>\frac{2}{3}$ when $n=3$) and $(\tau_0,B)$-standard surgeries on $(\varepsilon,k)$-hypersurface necks with $\varepsilon<\varepsilon_0$ and $k\ge k_0$. There exists $C=C(n,\alpha,V,\Theta)<\infty$ such that
\begin{align}\label{eq:gradient estimate surgery}
\vert\cd A\vert^2\le C(H^4+K^2) \quad\text{in}\quad \mc M_i^n\times[T_i,T_{i+1}]
\end{align}
for all $i$.
\end{theorem}
\end{comment}

Since the derivatives of the second fundamental form are zero on round Euclidean cylinders and spherical caps, the derivative estimates also pass to surgically modified flows.
 
\begin{theorem}[Gradient estimate for surgically modified flows (Cf. {\cite[Theorem 6.1]{HuSi09}})]\label{thm:gradient surgery}
Let $\{X_i:\M_i^n\times[T_i,T_{i+1}]\to\mathbb{S}_K^{n+1}\}_{i=1}^{N-1}$, $n\ge 3$, be a surgically modified flow with initial condition in the class $\mathcal{C}_K^n(\alpha,V,\Theta)$ (with $\alpha>\frac{2}{3}$ when $n=3$). There exists $C=C(n,\alpha,V,\Theta)<\infty$ such that
\begin{align}\label{eq:gradient estimate surgery}
\vert\cd A\vert^2\le C(H^4+K^2) \quad\text{in}\quad \mc M_i^n\times[T_i,T_{i+1}]
\end{align}
for all $i$.
\end{theorem}
\begin{proof}
We proceed as in the proof of Theorem \ref{thm:gradient estimate}, but with $\delta$ taken to be zero and fixed $\eta=\beta$, where $\beta$ is defined by \eqref{eq:beta in grad est}. First observe that, since $|A|^2-\frac{1}{n-1}H^2\equiv 0$ on a round cylinder in Euclidean space, we may choose a suitable range of neck and surgery parameters, and surgery scales, so that
\begin{align*}
|A|^2-\frac{1}{n-1} H^2 \leq \frac{\beta}{2}H^2
\end{align*}
on regions modified or added by surgery. We may therefore arrange that
\begin{align*}
G_\beta:={}& \left(\frac{1}{n-1}+\beta\right) H^2-|A|^2 + 2C_\beta K\geq \frac{\beta}{2} H^2 
\end{align*}
and 
\begin{align*}
G_0:={}&\frac{3}{n+2}H^2 -|A|^2 + 2C_0 K\geq \frac{3\beta}{2}H^2.
\end{align*}
Furthermore, since $|\nabla A|^2 \equiv 0$ on a round cylinder in Euclidean space, we may choose a suitable range of neck and surgery parameters, and surgery scales, so that, on regions modified or added by surgery, $|\nabla A|^2 \leq \mu_0 H^4$, where $\mu_0$ is a constant which depends only on $n$. Thus, in regions modified or added by surgery,
\begin{align*}
\frac{|\nabla A|^2}{G_0G_\beta} \leq \frac{4\mu_0}{3\beta^2}\,.
\end{align*} 
Since the surgically modified flow remains in a fixed surgery class, we may proceed as in the proof of Theorem \ref{thm:gradient estimate} in the time intervals $(T_i,T_{i+1})$. 
\begin{comment}
We can iterate the argument between each surgery time and get 
\begin{align}
\frac{|\nabla A|^2}{G_0G_\eta} \leq \max \left \{ \mu _0 \times \frac{2}{\eta} \times \frac{1} { \left ( \frac{3}{n+2} - \eta - \frac{1}{n-1} \right )}, \frac{ 3 N}{ 2 \kappa_n (n+2)}, m_0 \right \}. 
\end{align}
We can choose $\eta = \kappa_n$ in our definition of $G_\eta$ and therefore we see that $C_\eta = C(n,\alpha,V,\Theta)$ and hence
\begin{align*}
G_0G_\eta \leq H^4 + CK^2\,.
\end{align*}
The estimate above then becomes
\begin{align*}
|\nabla A|^2 \leq{}& C(n) |A|^4 + C(n,\alpha,V,\Theta)K^{2}.\qedhere
\end{align*}
\end{comment}
\end{proof}

\subsection{Higher order estimates}

Proceeding similarly as  in Theorem \ref{thm:gradient surgery} (cf. \cite[Theorem 6.3]{HuSi09}) yields estimates for higher derivatives of $A$ along surgically modified flows.

\begin{theorem}[Hessian estimate for surgically modified flows (cf. {\cite[Theorem 6.3]{HuSi09}})]\label{thm:Hessian estimate surgery}
Let $\{X_i:\M_i^n\times[T_i,T_{i+1}]\to\mathbb{S}_K^{n+1}\}_{i=1}^{N-1}$, $n\ge 3$, be a surgically modified flow with initial condition in the class $\mathcal{C}_K^n(\alpha,V,\Theta)$ (with $\alpha>\frac{2}{3}$ when $n=3$). There exists $C=C(n,\alpha,V,\Theta)$ such that
\ba\label{eq:scale_invariant_Hessian_estimate}
\vert\cd^2{A}\vert^2\leq C(H^{6}+K^{3})\;\;\text{in}\;\; \M \times[\lambda_0K^{-1},T)\,.
\ea
\end{theorem}
\begin{proof}
Proceed as in Theorem \ref{thm:Hessian estimate} between surgeries and use the fact that, for suitable neck and surgery parameters, and surgery scales, $\vert \cd^2 A\vert^2/H^6$ is small in regions modified or added by surgery.
\end{proof}

\begin{comment}
\begin{theorem}[Higher-order estimates]\label{thm:higher order estimates surgery}
For all $n\ge 3$, $K>0$ and $\Lambda>0$, there exist $\eta_0>0$, neck parameters $\varepsilon_0>0$ and $k_0\ge 2$, and surgery parameters $(\tau_0,B)$ with the following property. Let $\{X_i:\M_i^n\times[T_i,T_{i+1}]\to\mathbb{S}_K^{n+1}\}_{i=1}^{N-1}$ be a surgically modified flow with initial condition in the class $\mathcal{C}_K^n(\alpha,V,\Theta)$ (with $\alpha>\frac{2}{3}$ when $n=3$) and $(\tau_0,B)$-standard surgeries on $(\varepsilon,k)$-hypersurface necks with $\varepsilon<\varepsilon_0$ and $k\ge k_0$. There exist, for each pair of non-negative integers $k$ and $\ell$, constants $C_{k,\ell}=C_{k,\ell}(k,\ell,n,\alpha,V,\Theta)$ such that
\ba\label{eq:scale_invariant_higher_derivative_estimates}
\vert\cd_t^k\cd^\ell {A}\vert^2\leq C_{k,\ell}(H^{2+4k+2\ell}+K^{1+2k+\ell})\;\;\text{in}\;\; \M \times[\lambda_0K^{-1}\!,T).
\ea
\end{theorem}
\end{comment}

Analogues of the higher order estimates \eqref{eq:scale_invariant_higher_derivative_estimates} also pass to surgically modified flows, but, as mentioned above, they will not actually be needed in the construction.

\begin{comment}
\begin{theorem}[Higher-order estimates]\label{thm:higher order estimates surgery}
Let $\{X_i:\M_i^n\times[T_i,T_{i+1}]\to\mathbb{S}_K^{n+1}\}_{i=1}^{N-1}$, $n\ge 3$, be a surgically modified flow with initial condition in the class $\mathcal{C}_K^n(\alpha,V,\Theta)$ (with $\alpha>\frac{2}{3}$ when $n=3$). There exist, for each pair of non-negative integers $k$ and $\ell$, constants $C_{k,\ell}=C_{k,\ell}(k,\ell,n,\alpha,V,\Theta)$ such that
\ba\label{eq:scale_invariant_higher_derivative_estimates}
\vert\cd_t^k\cd^\ell {A}\vert^2\leq C_{k,\ell}(H^{2+4k+2\ell}+K^{1+2k+\ell})\;\;\text{in}\;\; \M \times[\lambda_0K^{-1}\!,T).
\ea
\end{theorem}
\end{comment}

\subsection{Neck detection}
 
%Given the cylindrical estimate and the gradient estimates, we can prove the Neck Detection Lemma. Essentially the theorem tells us that if one principal curvature is small relative the mean curvature and the curvature is sufficiently large and the parabolic neighbourhood is free of surgeries, then the point $(p_0,t_0)$ is the centre of a shrinking curvature neck where we will perform surgery.

The conclusion of the neck detection lemma \ref{lem:curvature neck detection} also holds for surgically modified flows, so long as we work in regions which are not affected by surgeries (cf. \cite[Lemma 7.4]{HuSi09}). %We need the following versions.

In the following theorem, a region $U\times I$ is \emph{free of surgeries} if at each surgery time $T_i\in I$, $i\in \{2,.\dots,N-1\}$, we have $U\subset \M_{i-1}\cap\M_i$ and $X_{i-1}|_U(\cdot,T_{i})=X_{i}|_U(\cdot,T_{i})$ (and hence $X_{i-1}$ and $X_i$ may be pasted together to form a smooth mean curvature flow in $U\times I$).

\begin{theorem}[{Neck detection for surgically modified flows (cf. \cite[Lemma 7.4]{HuSi09})}] \label{lem_NDL} 
Let $\{X_i:\M_i^n\times[T_i,T_{i+1}]\to\mathbb{S}_K^{n+1}\}_{i=1}^{N-1}$, $n\geq 3$, be a surgically modified flow with initial condition in the class $\mathcal C^n_K(\alpha,V,\Theta)$. Given $\varepsilon$, $\theta$, $L$ and $ k$, there exist positive $\eta_\sharp$, $h_\sharp$ with the following property: If
\begin{enumerate}
\item[(ND1)] $|H(p_0,t_0)| \geq h_\sharp \sqrt K$ and  $\frac{ \lambda_1(p_0,t_0)}{|H(p_0,t_0)|}\leq \eta_\sharp$, and
\item[(ND2)] the neighbourhood $\mathcal P\left(p_0,t_0,\frac{(n-1)(L+1)}{H(p_0,t_0)},\frac{\theta}{H^2(p_0,t_0)}\right)$ is free of surgeries,
\end{enumerate} 
then $(p_0,t_0)$ lies at the centre of an $(\varepsilon,k,L)$-neck.
%\begin{itemize}
%\item[(A)] the parabolic neighbourhood $\mathcal P (p_0,t_0,L,\theta)$ is an $ (\varepsilon,k_0-1,L,\theta)$ shrinking curvature neck, 
%\item[(B)] the parabolic neighbourhood $\mathcal P(p_0,t_0,L-1,\theta/2)$ is an $(\varepsilon,k,L-1, \theta/2)$ shrinking hypersurface neck.
%\end{itemize}
\end{theorem} 
\begin{proof}
The proof of Lemma \ref{lem:curvature neck detection} applies using Theorems \ref{thm:cylindrical surgery}, \ref{thm:gradient surgery} and  \ref{thm:Hessian estimate surgery} in lieu of Theorems \ref{thm:cylindrical estimate}, \ref{thm:gradient estimate} and \ref{thm:Hessian estimate}, due to the hypothesis (ND2).
\end{proof}

\section{Existence of terminating surgically modified flows}

We say that a surgically modified flow $\{X_i:\M_i^n\times[T_i,T_{i+1}]\to\mathbb{S}_K^{n+1}\}_{i=1}^{N-1}$ \emph{terminates} at the final time $T:=T_N<\infty$ if either
\begin{itemize}
%\item each connected component of $X_{N-1}(\M_{N-1},T_N)$ converges under mean curvature flow in infinite time to a totally geodesic hypersphere, or 
\item each connected component of $X_{N-1}(\M_{N-1},T_N)$ is diffeomorphic to $\mathbb S^n$ or to $\mathbb S^1\times\mathbb S^{n-1}$, or
\item after performing surgery on $X_{N-1}(\M_{N-1},T_N)$, each connected component of the resulting hypersurface is diffeomorphic to $\mathbb S^n$ or to $\mathbb S^1\times\mathbb S^{n-1}$.
\end{itemize}

\begin{theorem}[Existence of terminating surgically modified flows]
Let $X:\M\to \mathbb S_K^{n+1}$, $n\ge 3$, be a properly immersed hypersurface satisfying the quadratic pinching condition \eqref{eq:strict quadratic pinching}. There exists a surgically modified flow $\{X_i:\M_i^n\times[T_i,T_{i+1}]\to\mathbb{S}_K^{n+1}\}_{i=1}^{N-1}$ with $X_1(\cdot,0)=X$ which terminates at time $T=T_N$.
\end{theorem}
\begin{proof}
Given the cylindrical and gradient estimates, and the neck detection lemma, and a sufficiently small choice of the surgery scale $r$, we can proceed as in \cite[Section 8]{HuSi09} using the machinery developed in \cite[Sections 3 and 7]{HuSi09}, with only minor modifications required. These are:

1. In order to reconcile our data $\mathcal{C}_K^n(\alpha,V,\Theta)$ with those of \cite{HuSi09}, we replace the parameter $K$ by introducing the scale factor $R:=1/\sqrt{\Theta K}$. Our data $\alpha$ and $V$ can then be related to their $\alpha_0$ and $\alpha_2$, respectively. The constant $\alpha_1$ which appears in \cite{HuSi09} is not needed here. Since the surgery scale may be taken as small as needed, we may then choose the surgery parameters (albeit with slightly worse values) as explained in \cite[pp. 208--209]{HuSi09}.

2. Since our ambient space is non-Euclidean, the proof of the neck continuation theorem requires modification in two places. These are explained and carried out in detail in a more general setting in \cite[Section 8]{BrHu17}.

3. Since the maximal time is not a priori bounded in the present setting, the surgery algorithm may not terminate ``on its own''. Observe, however, that the maximum of the mean curvature must eventually drop permanently below the scale which triggers the surgery; indeed, if this were not the case, then an infinite number of surgeries would be carried out, an impossibility since we began with a finite amount of area, area is non-increasing under the flow, and each surgery decreases area by at least a certain fixed amount. The flow can then be smoothly continued indefinitely. Since the curvature remains uniformly bounded, standard arguments imply that each connected component converges, along some sequence of times approaching infinity, to a minimal hypersurface. The cylindrical estimate (Theorem \ref{thm:cylindrical estimate}) applied independently to each connected component then implies that each component of the limit is a totally geodesic hypersphere. So the flow must terminate afterall.
\begin{comment}
\noindent \textcolor{red}{\texttt{1. Need to be slightly careful about the parameters. First, replace the scale factor $K$ in the estimates by defining $R:=1/\sqrt{\Theta K}$. Substitute Corollay 6.5 (what about 7.2 and 7.19?) with our Lemma \ref{lem:integrate the gradient estimate (parabolic)}.\\
2. Need to modify the neck continuation as in \cite{BrHu17}.\\
3. Long time behaviour is only slightly different: if the flow doesn't terminate ``on its own'', then we must reach an initial datum on which all components have uniformly bounded mean curvature under mean curvature flow (else we perform further surgeries using up more area). By the cylindrical estimate and the Bernstein estimates, all components converge smoothly to ``big'' $S^n$'s.}}
\end{comment}
\end{proof}

\bibliographystyle{plain}
\bibliography{../bibliography}

\begin{thebibliography}{10}

\bibitem{AlencarDoCarmo}
Hil\'{a}rio Alencar and Manfredo do~Carmo.
\newblock Hypersurfaces with constant mean curvature in spheres.
\newblock {\em Proc. Amer. Math. Soc.}, 120(4):1223--1229, 1994.

\bibitem{An02}
B.~Andrews.
\newblock Positively curved surfaces in the three-sphere.
\newblock In {\em Proceedings of the {I}nternational {C}ongress of
  {M}athematicians, {V}ol. {II} ({B}eijing, 2002)}, pages 221--230. Higher Ed.
  Press, Beijing, 2002.

\bibitem{AndrewsRiemannian}
Ben Andrews.
\newblock Contraction of convex hypersurfaces in {R}iemannian spaces.
\newblock {\em J. Differential Geom.}, 39(2):407--431, 1994.

\bibitem{An12}
Ben Andrews.
\newblock Noncollapsing in mean-convex mean curvature flow.
\newblock {\em Geom. Topol.}, 16(3):1413--1418, 2012.

\bibitem{AnBa10}
Ben Andrews and Charles Baker.
\newblock Mean curvature flow of pinched submanifolds to spheres.
\newblock {\em J. Differential Geom.}, 85(3):357--395, 2010.

\bibitem{EGF}
Ben Andrews, Bennett Chow, Christine Guenther, and Mat Langford.
\newblock {\em {E}xtrinsic {G}eometric {F}lows}.
\newblock {G}raduate {S}tudies in {M}athematics, Vol. 206. {A}merican
  {M}athematical {S}ociety, 2020.

\bibitem{AndrewsHanLiWei}
Ben Andrews, Xiaoli Han, Haizhong Li, and Yong Wei.
\newblock Non-collapsing for hypersurface flows in the sphere and hyperbolic
  space.
\newblock {\em Ann. Sc. Norm. Super. Pisa Cl. Sci. (5)}, 14(1):331--338, 2015.

\bibitem{AL16}
Ben Andrews and Mat Langford.
\newblock Two-sided non-collapsing curvature flows.
\newblock {\em Ann. Sc. Norm. Super. Pisa Cl. Sci. (5)}, 15:543--560, 2016.

\bibitem{ALM13}
Ben Andrews, Mat Langford, and James McCoy.
\newblock Non-collapsing in fully non-linear curvature flows.
\newblock {\em Ann. Inst. H. Poincar\'e Anal. Non Lin\'eaire}, 30(1):23--32,
  2013.

\bibitem{BrHu17}
Simon Brendle and Gerhard Huisken.
\newblock A fully nonlinear flow for two-convex hypersurfaces in {R}iemannian
  manifolds.
\newblock {\em Invent. Math.}, 210(2):559--613, 2017.

\bibitem{BuzanoHaslhoferHershkovitsTori}
Reto Buzano, Robert Haslhofer, and Or~Hershkovits.
\newblock The moduli space of two-convex embedded tori.
\newblock {\em Int. Math. Res. Not. IMRN}, 2019(2):392--406, 2019.

\bibitem{BuzanoHaslhoferHershkovitsSphere}
Reto Buzano, Robert Haslhofer, and Or~Hershkovits.
\newblock The moduli space of two-convex embedded spheres.
\newblock Preprint,
  \href{https://arxiv.org/abs/1607.05604v2}{arXiv:1607.05604}, 2020.

\bibitem{ChengNakagawa}
Qing~Ming Cheng and Hisao Nakagawa.
\newblock Totally umbilic hypersurfaces.
\newblock {\em Hiroshima Math. J.}, 20(1):1--10, 1990.

\bibitem{ChernDoCarmoKobayashi}
S.~S. Chern, M.~do~Carmo, and S.~Kobayashi.
\newblock Minimal submanifolds of a sphere with second fundamental form of
  constant length.
\newblock In {\em Functional {A}nalysis and {R}elated {F}ields ({P}roc. {C}onf.
  for {M}. {S}tone, {U}niv. {C}hicago, {C}hicago, {I}ll., 1968)}, pages 59--75.
  Springer, New York, 1970.

\bibitem{HK1}
Robert Haslhofer and Bruce Kleiner.
\newblock Mean curvature flow of mean convex hypersurfaces.
\newblock {\em Comm. Pure Appl. Math.}, 70(3):511--546, 2017.

\bibitem{HK2}
Robert Haslhofer and Bruce Kleiner.
\newblock Mean curvature flow with surgery.
\newblock {\em Duke Math. J.}, 166(9):1591--1626, 2017.

\bibitem{HoSp74}
David {Hoffman} and Joel {Spruck}.
\newblock {Sobolev and isoperimetric inequalities for Riemannian submanifolds.}
\newblock {\em {Commun. Pure Appl. Math.}}, 27:715--727, 1974.

\bibitem{Hu84}
Gerhard Huisken.
\newblock Flow by mean curvature of convex surfaces into spheres.
\newblock {\em J. Differential Geom.}, 20(1):237--266, 1984.

\bibitem{Hu86}
Gerhard Huisken.
\newblock Contracting convex hypersurfaces in {R}iemannian manifolds by their
  mean curvature.
\newblock {\em Invent. Math.}, 84(3):463--480, 1986.

\bibitem{Hu87}
Gerhard Huisken.
\newblock Deforming hypersurfaces of the sphere by their mean curvature.
\newblock {\em Math. Z.}, 195(2):205--219, 1987.

\bibitem{HuSi09}
Gerhard Huisken and Carlo Sinestrari.
\newblock Mean curvature flow with surgeries of two-convex hypersurfaces.
\newblock {\em Invent. Math.}, 175(1):137--221, 2009.

\bibitem{HuSi15}
Gerhard Huisken and Carlo Sinestrari.
\newblock Convex ancient solutions of the mean curvature flow.
\newblock {\em J. Differential Geom.}, 101(2):267--287, 2015.

\bibitem{MiSi73}
J.~H. Michael and L.~M. Simon.
\newblock Sobolev and mean-value inequalities on generalized submanifolds of
  {$R^{n}$}.
\newblock {\em Comm. Pure Appl. Math.}, 26:361--379, 1973.

\bibitem{Mramor}
Alexander Mramor.
\newblock A finiteness theorem via the mean curvature flow with surgery.
\newblock {\em J. Geom. Anal.}, 28(4):3348--3372, 2018.

\bibitem{Ng15}
Huy~The {Nguyen}.
\newblock {Convexity and cylindrical estimates for mean curvature flow in the
  sphere.}
\newblock {\em {Trans. Am. Math. Soc.}}, 367(7):4517--4536, 2015.

\bibitem{Nguyen2020}
Huy~The Nguyen.
\newblock High codimension mean curvature flow with surgery.
\newblock \href{https://arxiv.org/abs/2004.07163}{arXiv:2004.07163}, 2020.

\bibitem{Okumura}
Masafumi Okumura.
\newblock Hypersurfaces and a pinching problem on the second fundamental
  tensor.
\newblock {\em Amer. J. Math.}, 96:207--213, 1974.

\bibitem{Si68}
James Simons.
\newblock Minimal varieties in riemannian manifolds.
\newblock {\em Ann. of Math. (2)}, 88:62--105, 1968.

\bibitem{St66}
Guido Stampacchia.
\newblock {\em \`{E}quations elliptiques du second ordre \`a coefficients
  discontinus}.
\newblock S\'eminaire de Math\'ematiques Sup\'erieures, No. 16 (\'Et\'e, 1965).
  Les Presses de l'Universit\'e de Montr\'eal, Montreal, Que., 1966.

\end{thebibliography}

\end{document}